\documentclass{jocg}
\usepackage[textwidth=6.6in]{geometry} 
\usepackage{enumitem}
\usepackage{amsmath, amssymb, amsthm}

\usepackage{mathtools}
\usepackage{geometry}
\usepackage{array}
\usepackage{hyperref}
\usepackage[square,sort,comma,numbers]{natbib}

\usepackage{color}
\definecolor{darkred}{RGB}{100,0,0}
\definecolor{darkgreen}{RGB}{0,100,0}
\definecolor{darkblue}{RGB}{0,0,150}
\definecolor{citecol}{RGB}{30,80,150}
\definecolor{tabcol}{RGB}{200,230,255}

\usepackage{colortbl}

\usepackage{hyperref}
\hypersetup{colorlinks=true, linkcolor=darkred, citecolor=citecol, urlcolor=darkblue}
\usepackage{url}

\allowdisplaybreaks
\newtheorem{theorem}{Theorem}[section]
\newtheorem{corollary}[theorem]{Corollary}
\newtheorem{prop}[theorem]{Proposition}
\newtheorem{lemma}[theorem]{Lemma}
\newtheorem{definition}[theorem]{Definition}
\theoremstyle{remark}

\newtheorem{remark}[theorem]{Remark}

\newcommand{\R}{\mathbb{R}}
\newcommand{\N}{\mathbb{N}}

\newcommand{\X}{\mathbb{X}}
\newcommand{\Y}{\mathbb{Y}}

\newcommand{\W}{\mathbb{W}}

\newcommand{\DD}{\mathcal{D}}

\renewcommand{\P}{\mathbb{P}}

\newcommand{\RR}{\mathcal{R}}

\newcommand{\KK}{\mathcal{K}}

\newcommand{\CC}{\mathcal{C}}

\newcommand{\dd}{\mathrm{d}}

\newcommand{\p}[1]{\left(#1 \right)}

\DeclareMathOperator*{\id}{id}

\newcommand{\upperdiag}{\R^2_>}
\newcommand{\dgm}{D}
\DeclareMathOperator{\dgmmap}{dgm}  
\newcommand{\rD}{\mathrm{r}}
\DeclareMathOperator{\Pers}{Pers}  
\DeclareMathOperator{\pers}{pers} 
\DeclareMathOperator{\Lip}{Lip}  
\DeclareMathOperator{\diam}{diam}  
\newcommand{\ones}{\mathbf{1}}
\newcommand{\defeq}{\vcentcolon =}

\numberwithin{equation}{section}

\begin{document}
\raggedbottom

\title{\MakeUppercase{On the choice of weight functions for linear representations of persistence diagrams}}

\author{Vincent Divol  \thanks{\affil{Inria Saclay and Universit\'e Paris-Sud},
          \email{vincent.divol@inria.fr}}\ \&\ Wolfgang Polonik \thanks{\affil{University of California, Davis}, \email{wpolonik@ucdavis.edu} }
}

\date{\today}

\maketitle

\begin{abstract}
Persistence diagrams are efficient descriptors of the topology of a point cloud. As they do not naturally belong to a Hilbert space, standard statistical methods cannot be directly applied to them. Instead, feature maps (or representations) are commonly used for the analysis. A large class of feature maps, which we call \emph{linear}, depends on some weight functions, the choice of which is a critical issue. An important criterion to choose a weight function is to ensure stability of the feature maps with respect to Wasserstein distances on diagrams. We improve known results on the stability of such maps, and extend it to general weight functions. We also address the choice of the weight function by considering an asymptotic setting; assume that $\X_n$ is an i.i.d.\ sample from a density on $[0,1]^d$. For the \v Cech and Rips filtrations, we characterize the weight functions for which the corresponding feature maps converge as $n$ approaches infinity, and by doing so, we prove laws of large numbers for the total persistences of such diagrams.  Those two approaches (stability and convergence) lead to the same simple heuristic for tuning weight functions: if the data lies near a $d$-dimensional manifold, then a sensible choice of weight function is the persistence to the power $\alpha$ with $\alpha \geq d$.
\end{abstract}

\section{Introduction}

Topological data analysis, or TDA (see \cite{chazal2017introduction} for a survey) is a recent field at the intersection of computational geometry, statistics and probability theory that has been successfully applied to various scientific areas, including biology \cite{yao2009topological}, chemistry \cite{nakamura2015persistent}, material science \cite{lee2017quantifying} or the study of time series \cite{seversky2016time}. It consists of an array of techniques aimed at understanding the topology of a $d$-dimensional manifold based on an approximating point cloud $\X$. For instance, clustering can be seen as the estimation of the connected components of a given manifold. Persistence diagrams are one of the tools used most often in TDA. They are efficient descriptors of the topology of a point cloud, consisting in a multiset $\dgm$  of points in $\upperdiag \defeq \{\rD=(r_1,r_2)\in \R^2, r_1<r_2\}$ (see Section \ref{sec:background} for a more precise definition). The space $\mathcal{D}$ of persistence diagrams is not naturally endowed with a Hilbert or Banach space structure, making statistical inference rather awkward. A common scheme to overcome this issue is to use a \emph{representation} or \emph{feature map} $\Phi : \mathcal{D} \to \mathcal{B}$, where $\mathcal{B}$ is some Banach space: classical machine learning techniques are then applied to $\Phi(\dgm)$ instead of $\dgm$, where it is assumed that an entire set (or sample) of persistence diagrams is observed. A natural way to create such feature maps is to consider a function $\phi : \upperdiag \to \mathcal{B}$ and to define
\begin{equation}\label{eq:linear_representation}
\Phi(\dgm) \defeq \sum_{\rD \in \dgm} \phi(\rD).
\end{equation}
A multiset can equivalently be seen as a measure. Therefore we let $\dgm$ also denote the measure $\sum_{\rD\in \dgm}\delta_{\rD}$ with $\delta_\rD$ denoting Dirac measure in $\rD$. With this notation, $\Phi(\dgm)$ is equal to $\dgm(\phi)$, the integration of $\phi$ against the measure $\dgm$. Representations as in \eqref{eq:linear_representation} are called \emph{linear} as they define linear maps from the space of finite signed measures to the Banach space $\mathcal{B}$. In the following, a representation will always be considered linear. Many linear representations exist in the literature, including persistence surfaces and its variants \cite{chen2015statistical,reininghaus2015stable,kusano2018kernel,adams2017persistence}, persistence silhouettes \cite{chazal2014stochastic} or accumulated persistence function \cite{biscio2016accumulated}. Notable non-linear representations inlude persistence landscapes \cite{bubenik2015statistical}, and sliced Wasserstein kernels \cite{carriere2017sliced}. 

In machine learning, a possible way to circumvent the so-called "curse of dimensionality" is to assume that the data lies near some low-dimensional manifold $M$. Under this assumption, the persistence diagram of the data set (built with the \v Cech filtration, for instance) is made of two different types of points: points $\dgm_{\text{true}}$ far away from the diagonal, which estimate the diagram of the manifold $M$, and points $\dgm_{\text{noise}}$ close to the diagonal, which are generally considered to be "topological noise" (see Figure \ref{fig:ex_pd}). 
This interpretation is a consequence of the stability theorem for persistence diagrams; see \cite{cohen2007stability}. If the relevant information lies in the structure of the manifold, then the topological noise indeed represents true noise, and representations of the form $\dgm(\phi)$ are bound to fail if $\dgm_{\text{noise}}(\phi)$ is dominating $\dgm_{\text{true}}(\phi)$. A way to avoid such behaviour is to weigh the points in diagrams by means of a weighting function $w:\upperdiag\to \R$. If $w$ is chosen properly, i.e.\ small enough when close to the diagonal, then one can hope that $\dgm_{\text{true}}(w\phi)$ can be separated from $\dgm_{\text{noise}}(w\phi)$. The weight functions $w$ are typically chosen as functions of the persistence $\pers(\rD) \defeq r_2-r_1$, a choice which will be made here also. Of course, it is not clear what "small enough" really means, and there are several ways to address the issue.

A first natural answer is to look at the problem from a stability point of view. Indeed, as data are intrinsically noisy, a statistical method has to be stable with respect to some metric in order to be meaningful. Standard metrics on the space of diagrams $\mathcal{D}$ are Wasserstein distances $W_p$, which under mild assumptions (see \cite{cohen2010lipschitz}) are known to be stable with respect to the data on which diagrams are built. The task therefore becomes to find representations $\dgm(w\phi)$ that are continuous with respect to some Wasserstein distance. Recent work in \cite{kusano2018kernel} shows that when sampling from a $d$-dimensional manifold, a weight function of the form $w(\rD)= \arctan(A \cdot \pers(\rD)^\alpha)$ with $\alpha >d+1$ ensures that a certain class of representations are Lipschitz. Our first contribution is to show that, for a general class of weight functions,  a choice of $\alpha >d$ is enough to make \emph{all} linear representations continuous (even H\"olderian of exponent $\alpha-d$).

\begin{figure}
\includegraphics[width= \textwidth]{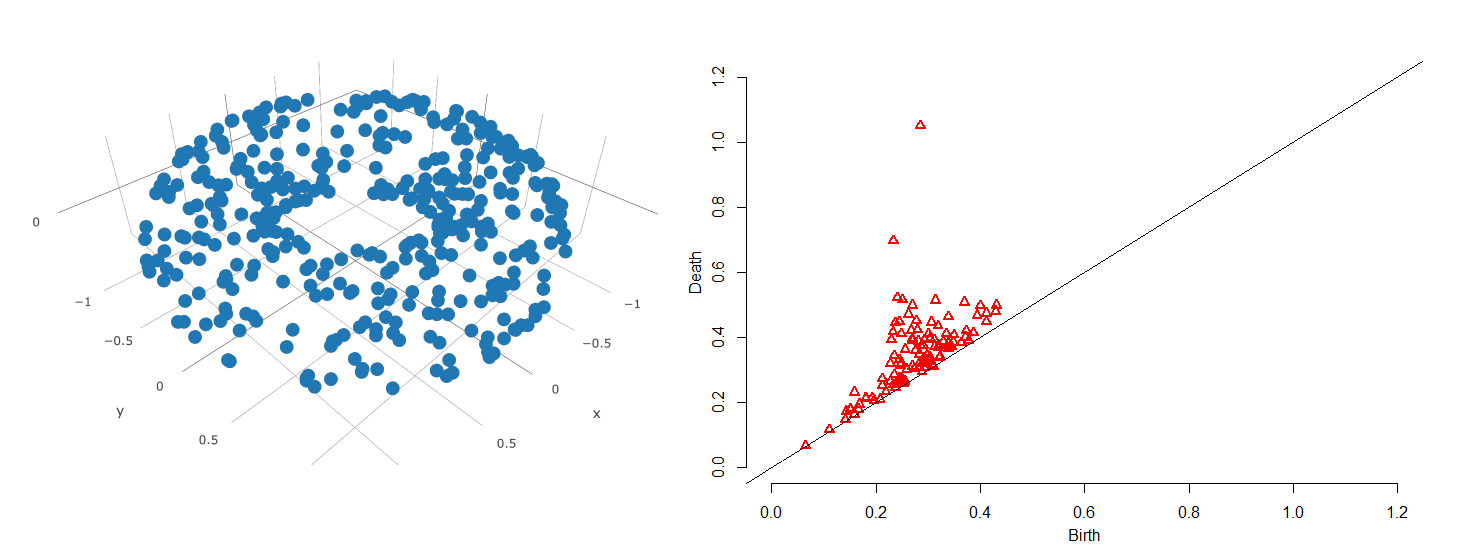}
\caption{The persistence diagram for homology of degree 1 of the Rips filtration of $2000$ i.i.d.\ points uniformly sampled on a torus. There are two distinct points in the diagram, corresponding to the two equivalence classes of one-dimensional holes of the torus.}
\label{fig:ex_pd}
\end{figure}

Our second (and main) contribution is to evaluate closeness to the diagonal from an asymptotic point of view. Assume that a diagram $\dgm_n$ is built on a data set of size $n$. For which weight functions is $\dgm_{n,\text{noise}}(w\phi)$ none-divergent? Of course, for this question to make sense, a model for the data set has to be specified. A simple model is given by a Poisson (or binomial) process $\X_n$ of intensity $n$ in a cube of dimension $d$. We denote the corresponding diagrams built on a filtration $\KK$ with respect to $q$-dimensional homology by $\dgmmap_q[\KK(\X_n)]$, with $\KK$ either the Rips or \v Cech filtration. A precise definition is given below in Section \ref{sec:background}. In this setting, there are no "true" topological features (other than the trivial topological feature of $[0,1]^d$ being connected), and thus the diagram based on the sampled data is uniquely made of topological noise. A first promising result is the vague convergence of the measure $\mu_q^n\defeq n^{-1}\dgmmap[\KK(n^{1/d}\X_n)]$, which was recently proven in \cite{hiraoka2018limit} for homogeneous Poisson processes in the cube and in \cite{goel2018strong} for binomial processes on manifolds. However, vague convergence is not enough for our purpose, as neither $\phi$ nor $w$ have good reasons to have compact support. Our main result, Theorem \ref{thm:main_cv} extends result of \cite{goel2018strong}, for processes on the cube, to a stronger convergence, allowing test functions to have both non-compact support (but to converge to $0$ near the diagonal) and to have polynomial growth. As a corollary of this general result, the convergence of the $\alpha$-th total persistence, which plays an important role in TDA, is shown. The $\alpha$-th total persistence is defined as $\Pers_\alpha(\dgm) \defeq \dgm(\pers^\alpha) = \sum_{\rD\in \dgm} \pers(\rD)^\alpha$.

\begin{theorem}\label{thm:main_thm} Let $\alpha>0$ and let $\kappa$ be a density on $[0,1]^d$ such that $0<\inf \kappa \leq \sup \kappa < \infty$. Let $\X_n$ be either a binomial process with parameters $n$ and $\kappa$ or a Poisson process of intensity $n\kappa$ in the cube $[0,1]^d$. Define $\dgmmap_q[\KK(\X_n)]$ to be the persistence diagram of $\X_n$ for $q$-dimensional homology, built with either the Rips or the \v Cech filtration. Then, with probability one, as $n \to \infty$
\begin{equation}
n^{\frac{\alpha}{d}-1}\Pers_\alpha(\dgmmap_q[\KK(\X_n)]) \to \mu_q^\kappa(\pers^\alpha) <\infty
\end{equation}
for some non-degenerate Radon measure $\mu_q^\kappa$ on $\upperdiag$.
\end{theorem}
If $D_n \defeq \dgmmap_q[\KK(\X_n')]$ is built on a point cloud $\X_n'$ of size $n$ on a $d$-dimensional manifold, one can expect $D_{n,\text{noise}}$ to behave in a similar fashion to that of $\dgmmap_q[\KK(\X_n)]$ for $\X_n$ a $n$-sample on a $d$-dimensional cube (a manifold looking locally like a cube). Therefore, for $\alpha>0$, the quantity $\dgm_{n,\text{noise}}(\pers^\alpha)$ should be close to $\Pers_\alpha(\dgmmap[\KK(\X_n)]),$  and it can be expected to converge to $0$ if and only if the weight function $\pers^\alpha$ is such that $\alpha \geq d$. The same heuristic is found through both the approaches (stability and convergence): 
\textit{a weight function of the form $\pers^\alpha$ with $\alpha\geq d$ is sensible if the data lies near a $d$-dimensional object}.


Further properties of the process $(\dgmmap_q[\KK(\X_n)])_n$ are also shown, namely non-asymptotic rates of decays for the number of points in said diagrams, and the absolute continuity of the marginals of $\mu_q^\kappa$ with respect to the Lebesgue measure on $\R$.

\subsection{Related work}

Techniques used to derive the large sample results indicated above are closely related to the field of geometric probability, which is the study of geometric quantities arising naturally from point processes in $\R^d$. A classical result in this field, see \cite{steele1988growth}, proves the convergence of the total length of the minimum spanning tree built on $n$ i.i.d.\ points in the cube. This pioneering work can be seen as a $0$-dimensional special case of our general results about persistence diagrams built for homology of dimension $q$. This type of result has been extended to a large class of functionals in the works of J. E. Yukich and M. Penrose (see for instance \cite{mcgivney1999asymptotics,yukich2000asymptotics,penroseLLN} and \cite{penrose2003random} or \cite{yukich2006probability} for monographs on the subject). 

The study of higher dimensional properties of such processes is much more recent. Known results include convergence of Betti numbers for various models and under various asymptotics (see  \cite{Kahle2011,kahle2013limit,Yogeshwaran2017,bobrowski2017random}). The paper \cite{bobrowski2015maximally} finds bounds on the persistence of cycles in random complexes, and \cite{hiraoka2018limit} proves limit theorems for persistence diagrams built on homogeneous point processes. The latter is extended to non-homogeneous processes in \cite{trinh2017remark}, and to processes on manifolds in \cite{goel2018strong}. Note that our results constitute a natural extension of \cite{trinh2017remark}. In \cite{skraba2017randomly}, higher dimensional analogs of minimum spanning trees, called minimal spanning acycles, were introduced. Minimal spanning acycles exhibits strong links with persistence diagrams and our main theorem can be seen as a convergence result for weighted minimal spanning acycle on geometric random complexes. \cite{skraba2017randomly} also proves the convergence of the total $1$-persistence for Linial-Meshulam random complexes, which are models of random simplicial complexes of a combinatorial nature rather than a geometric nature.

\subsection{Notation}
\begin{itemize}
	\item[] $\|\cdot \|$\hspace*{0.88cm} Euclidean distance on $\R^d$.
	\item[] $\|\cdot \|_\infty$\hspace*{0.68cm} supremum-norm of a function.
	\item[] $B(x,r)$\hspace*{0.55cm} open ball of radius $r$ centered at $x$.
	\item[] diam$(\X)$\hspace*{0.35cm} diameter of a set $\X\subset \R^d$, defined as $\sup_{x,y\in \X}\|x-y\|$.
	\item[] $|\cdot |$ \hspace*{1.05cm}  total variation of a measure.
	\item[] $\#$\hspace*{1.4cm} cardinality of a set.
	\item[] $\Lip(f)$\hspace*{0.65cm} Lipschitz constant of a Lipschitz function $f$.
\end{itemize}

\paragraph{}The rest of the paper is organized as follows. In Section \ref{sec:background}, some background on persistent homology is briefly described. The stability results are then discussed in Section \ref{sec:stability} whereas the convergence results related to the asymptotic behavior of the sample-based linear representations are stated in Section \ref{sec:convergence}. Section \ref{sec:discussion} presents some discussion. Proofs can be found in Section \ref{sec:proofs}.

\section{Background on persistence diagrams}\label{sec:background}
Persistent homology deals with the evolution of homology through a sequence of topological spaces. We use the field of two elements $\mathbb{F}_2$ to build the homology groups. A filtration $\KK = (K^r)_{r \geq 0}$ is an increasing right-continuous sequence of topological spaces : $K^{r'}\subset K^r$ iff $r'\leq r$ and $K^r = \bigcap_{r'<r}K^{r'}$. For any $q\geq 0$, the inclusion of spaces give rise to linear maps between corresponding homology groups $H_q(K^r)$. The persistence diagram $\dgmmap_q[\KK]$ of the filtration is a succinct way to summarize the evolution of the homology groups. It is a multiset of points in $\upperdiag = \{\rD=(r_1,r_2)\in \R^2, r_1<r_2\}$\footnote{Persistence diagrams are in all generality multiset of points in $\{\rD=(r_1,r_2), -\infty \leq r_1<r_2\leq \infty\}$. We only consider diagrams which do not contain points "at infinity" throughout the paper.}, so that each point $\rD=(r_1,r_2)$ corresponds informally to a $q$-dimensional "hole" in the filtration $\KK$ that appears (or is born) at $r_1$ and disappears (or dies) at $r_2$. The persistence $\pers(\rD)$ of $\rD$ is defined as $r_2-r_1$ and is understood as the lifetime of the corresponding hole. Persistence diagrams are known to exist given mild assumptions on the filtration (see \cite[Section 3.8]{chazal2016structure}). Some basic descriptors of persistence diagrams include the $\alpha$-th total persistence of a diagram, defined as
\begin{equation}
\Pers_\alpha(\dgm) \defeq \dgm(\pers^\alpha) = \sum_{\rD \in \dgm} \pers(\rD)^\alpha,\ \alpha>0,
\end{equation}
and the persistent Betti numbers, defined as
\begin{equation}
\beta^{r,s}(\dgm) \defeq \dgm([0,r]\times (s,\infty)) = \sum_{\rD \in \dgm} \ones\{\rD \in [0,r]\times (s,\infty)\}, \ 0\leq r\leq s.
\end{equation}
Also, for $M\geq 0$, define
\begin{equation}
\Pers_\alpha(\dgm,M) \defeq \dgm(\pers^\alpha \ones\{ \pers\geq M\}).
\end{equation}
Given a subset $\X$ of a metric space $(\Y,d)$, standard constructions of filtrations are the \v Cech filtration $\CC(\X)=(C^r(\X))_{r\geq 0}$ and the Rips filtration $\RR(\X)=(R^r(\X))_{r\geq 0}$:
\begin{align}
C^r(\X) &= \left\{\mbox{finite } \sigma \subset \X, \ \bigcap_{x\in \sigma} B(x,r) \neq \emptyset\right\} \text{ and }\\
R^r(\X) &= \left\{\mbox{finite }\sigma \subset \X,\ \mbox{diam}(\sigma)\leq r \right\},
\end{align}
where the abstract simplicial complexes on the right are identified with their geometric realizations. The dimension of a simplex $\sigma$ is equal to $\#\sigma -1$. If $K$ is a simplicial complex, the set of its simplexes of dimension $q$ is denoted by $K_q$.

The space of persistence diagrams $\DD$ is the set of all finite multisets in $\upperdiag$. Wasserstein distances are standard distances on $\DD$. For $p\geq 1$, they are defined as:
\begin{equation}
 W_p(\dgm,\dgm') \defeq \min_{\gamma} \left(\sum_{\rD\in \dgm\cup \Delta} \|\rD-\gamma(\rD)\|^p\right)^{1/p},
\end{equation} 
where $\Delta$ is the diagonal of $\R^2$ and $\gamma : \dgm\cup \Delta \to \dgm'\cup \Delta$ is a bijection. The definition is extended to $p=\infty$ by 
\begin{equation}
W_\infty(\dgm,\dgm') = \min_{\gamma} \max_{\rD\in \dgm\cup \Delta} \|\rD-\gamma(\rD)\|,
\end{equation}
which is called the \emph{bottleneck distance}.

The use of Wasserstein distances is motivated by crucial stability properties they satisfy. Let $f,g:\X\to \R$ be two continuous functions on a triangulable space $\X$. Assuming that the persistence diagrams $\dgmmap_q[f]$ and $\dgmmap_q[g]$ of the filtrations defined by the sublevel sets of $f$ and $g$ exist and are finite (a condition called \emph{tameness}\footnote{Tameness holds under simple conditions, see \cite[Section 3.9]{chazal2016structure}, which we will always assume to hold in the following)}, the stability property of \cite[Main Theorem]{cohen2007stability} asserts that $W_\infty(\dgmmap_q[f],\dgmmap_q[g])\leq \|f-g\|_{\infty}$, i.e.\ the diagrams are stable with respect to the functions they are built with. The functions $f$ and $g$ have to be thought of as representing the data: for instance, if the \v Cech filtration is built on a data set $\X_n = \{X_1,\dots,X_n\}$, then $\dgmmap_q[\CC(\X_n)]=\dgmmap_q[f]$ where $f$ is the distance function to $\X_n$, i.e.\ $f(\cdot)=d(\cdot,\X_n)$. When $p<\infty$, similar stability results have been proved under more restrictive conditions on the ambient space $\X$, which we now detail. 

\begin{definition}\label{def:bounded_pers} A metric space $\X$ is said to have \emph{bounded $m$-th total persistence} if there exists a constant $C_{\X,m}$ such that for all tame 1-Lipschitz functions $f :\X \to \R,$ and, for all $q\geq 1$, $\Pers_m(\dgmmap_q[f])\leq C_{\X,m}$.
\end{definition}

This assumption holds, for instance, for a $d$-dimensional manifold $\X$ when $m >d$ with 
\begin{equation}
C_{\X,m} = \frac{m}{m-d}C_{\X}\, \diam(\X)^{m-d},
\end{equation}
$C_{\X}$ being a constant depending only on $\X$ (see \cite{cohen2010lipschitz}). The stability theorem for the $p$-th Wasserstein distances claims:

\begin{theorem}[Section 3 of \cite{cohen2010lipschitz}]\label{thm:wass_stab}
Let $\X$ be a compact triangulable metric space with  bounded $m$-th total persistence for some $m \geq 1$. Let $f,g :\X\to \R$ be two tame Lipschitz functions. Then, for $q\geq 0$,
\begin{equation}
W_p(\dgmmap_q[f],\dgmmap_q[g])\leq C_0^{\frac{1}{p}}\|f-g\|_{\infty}^{1-\frac{m}{p}},
\end{equation}
for $p\geq m$, where $C_0=C_{\X,m}\max\{\Lip(f)^{m},\Lip(g)^{m}\}$.
\end{theorem}

\section{Stability results for linear representations}\label{sec:stability}

In \cite[Corollary 12]{kusano2018kernel}, representations of diagrams are shown to be Lipschitz with respect to the $1$ Wasserstein distance for weight functions of the form $w(\rD) = \arctan(B \cdot \pers(\rD)^\alpha)$ with $\alpha> m+1,B>0$, provided the diagrams are built with the sublevels of functions defined on a space $\X$ having bounded $m$-th total persistence. The stability result is proved for a particular function $\phi:\upperdiag \to \mathcal{B}$ defined by $\phi(\rD) = K(\rD,\cdot)$, with $K$ a bounded Lipschitz kernel and $\mathcal{B}$ the associated RKHS (short for Reproducing Kernel Hilbert Space, see \cite{aronszajn1950theory} for a monograph on the subject). We present a generalization of the stability result to (i) general weight functions $w$, (ii) any bounded Lipschitz function $\phi,$ and (iii) we only require $\alpha >m$.

Consider weight functions $w:\upperdiag \to \R_+$ of the form $w(\rD) = \tilde{w}(\pers(\rD))$ for a differentiable function $\tilde{w}:\R_+ \to \R_+$ satisfying $\tilde{w}(0)=0$, and, for some $A>0$, $\alpha \geq 1$,
\begin{equation}\label{eq:grad_weight}
 \forall u\geq 0,\ |\tilde{w}'(u)| \leq Au^{\alpha-1}.
\end{equation}
Examples of such functions include $w(\rD) = \arctan(B \cdot \pers(\rD)^{\alpha})$ for  $B>0$ and $w(\rD)=\pers(\rD)^{\alpha}$. We denote the class of such weight functions by  $\W(\alpha,A)$. In contrast to \cite{kusano2018kernel}, the function $\phi$ does not necessarily take its values in a RKHS, but simply in a Banach space (so that its Bochner integral --see for instance \cite[Chapter 4]{diestel1984sequences}-- is well defined). 

\begin{theorem}\label{thm:main_stab}
Let $(\mathcal{B},\|\cdot \|_{\mathcal{B}})$ be a Banach space, and let $\phi: \upperdiag\to \mathcal{B}$ be a Lipschitz continuous function. Furthermore, for $w\in \W(\alpha,A)$ with $A>0,\, \alpha \geq 1,$ let $\Phi_w(\dgm) \defeq \dgm(w\phi),$ and for two persistence diagrams $\dgm_1$ and $\dgm_2$ let $G\{t\} \defeq \max\{\Pers_t(\dgm_1),\Pers_t(\dgm_2)\},\ t \ge 0$. Then, for $1\leq p \leq \infty,$ and $a\in [0,1]$ (and using the conventions $0/\infty = 0$ and $\infty/\infty = 1$), we have
\begin{align}
\|\Phi_w(\dgm_1)-\Phi_w(\dgm_2)\|_{\mathcal{B}} &\leq \Lip(\phi)\frac{A}{\alpha} \left(G\left\{p\frac{\alpha}{p-1}\right\}\right)^{1-\frac{1}{p}} W_p(\dgm_1,\dgm_2) \label{eq:ineq_stab_general}\\
& +\|\phi\|_\infty  A \left(2G\left\{p\frac{\alpha-a}{p-a}\right\}\right)^{1-\frac{a}{p}} W_p(\dgm_1,\dgm_2)^a. \nonumber
\end{align}
\end{theorem}
The quantity $G\{q\}$ can often be controlled. For instance, if the diagrams are built with Lipschitz continuous functions $f:\X\to \R,$ and $\X$ is a space having bounded $m$-th total persistence.
\begin{corollary}\label{cor:stab_bis}
Let $A>0,\, \alpha \geq 1, q\geq 0$ and consider a compact triangulable metric space $\X$ having bounded $m$-th total persistence for some $m \geq 1$. Suppose that $f,g : \X \to \R$ are two tame Lipschitz continuous functions, $w\in \W(\alpha,A)$, and $a\in [0,1]$. Then, for $m\leq p \leq \infty$ such that $\alpha \geq m+a\left(1-\frac{m}{p}\right) \geq 0$, if $C_0 = C_{\X,m}\max\{\Lip(f)^{m},\Lip(g)^{m}\}$ and $\ell$ is the maximum persistence in the two diagrams $\dgmmap_q[f], \dgmmap_q[g]$:

\begin{align}
&\|\Phi_w(\dgmmap_q[f])-\Phi_w(\dgmmap_q[g])\|_{\mathcal{B}}  \leq C_1 W_p(\dgmmap_q[f],\dgmmap_q[g]) +C_2 W_p(\dgmmap_q[f],\dgmmap_q[g])^a,\label{eq:ineq_main_stab}
\end{align}
where $C_1=\Lip(\phi) \frac{A}{\alpha}\ell^{\alpha-m\left(1-\frac{1}{p}\right)} C_0^{1-\frac{1}{p}}$ and $C_2 = \|\phi\|_\infty A\ell^{\alpha-m-a\left(1-\frac{m}{p}\right)}(2C_0)^{1-\frac{a}{p}}$.
\end{corollary}

If $\alpha > m+1$ and $p=\infty$, then the result is similar to Theorem 3.3 in \cite{kusano2018kernel}. However, Corollary \ref{cor:stab_bis} implies that the representations are still continuous (actually H\"older continuous) when $\alpha \in (m,m+1]$, and this is the novelty of the result. Indeed, for such an $\alpha$, one can always chose $a$ small enough so that the stability result \eqref{eq:ineq_main_stab} holds. The proofs of Theorem \ref{thm:main_stab} and Corollary \ref{cor:stab_bis} consist of adaptations of similar proofs in \cite{kusano2018kernel}. They can be found in Section \ref{sec:proofs}.

\begin{remark} $\;$
%
{\em (a)} One cannot expect to obtain an inequality of the form \eqref{eq:ineq_stab_general} without quantities $G\{t\}$ (or other quantities depending on the diagrams) appearing on the right-hand side. For instance, in the case $p=\infty$, it is clear that adding an arbitrary number of points near the diagonal will not change the bottleneck distance between the diagram, whereas the distance between representations can become arbitrarily large.\\[5pt]
%
{\em (b)} Laws of large numbers stated in the next section (see also Theorem \ref{thm:main_thm} already stated in the introduction), show that Theorem \ref{thm:main_stab} is optimal: take $w=\pers^\alpha$ and $\phi \equiv 1$. If $\X_n$ is a sample on the $d$-dimensional cube $[0,1]^d$ (which has bounded $m$-th total persistence for $m>d$), then $\Phi_w(\dgmmap_q[\CC(\X_n)])=\Pers_\alpha(\dgmmap_q[\CC(\X_n)])$. The quantity $\Pers_\alpha(\dgmmap_q[\CC(\X_n)])$ does not converge to $0$ for $\alpha \leq d$ (it even diverges if $\alpha <d$), whereas the bottleneck distance between $\dgmmap_q[\CC(\X_n)]$ and the empty diagram does converge to $0$.
\end{remark}

The following corollary to the stability result also is a contribution to the asymptotic study of next section. It presents rates of convergence of representations in a random setting. Let $\X_n=\{X_1,\dots,X_n\}$ be a $n$-sample of i.i.d.\ points from a distribution on some manifold $\X$. We are interested in the convergence of representations $\Phi_w(\dgmmap_q[\CC(\X_n)])$ to the representations $\Phi_w(\dgmmap_q[\CC(\X)])$. The nerve theorem asserts that for any subspace $\X' \subset \X$, $\dgmmap_q[\CC(\X')] = \dgmmap_q[f]$ where $f(x)$ is the distance from $x\in \X$ to $\X'$. We obtain the following corollary, whose proof is found in Section \ref{sec:proofs}:

\begin{corollary}\label{cor:first_asymptotics} Consider a $d$-dimensional compact Riemannian manifold $\X$, and let $\X_n=\{X_1,\dots,X_n\}$ be a $n$-sample of i.i.d.\ points from a distribution having a density $\kappa$ with respect to the $d$-dimensional Hausdorff measure on $\X$. Assume that $0<\inf \kappa \leq \sup \kappa <\infty$. Let $w\in \W(\alpha,A)$ for some $A>0, \alpha>d,$ and let $\phi: \upperdiag \to \mathcal{B}$ be a Lipschitz function. Then, for $q\geq 0$, and for $n$ large enough,
\begin{equation}
E\left[ \|\Phi_w(\dgmmap_q[\CC(\X_n)])-\Phi_w(\dgmmap_q[\CC(\X)])\|_{\mathcal{B}}\right] \leq C\|\phi\|_\infty \frac{ \alpha }{\alpha-d} \left(\frac{\ln n}{n}\right)^{\frac{\alpha}{d}-1},
\end{equation}
where $C$ is a constant depending on $\X, A$ and the density $\kappa$.
\end{corollary}

\section{Convergence of total persistence}\label{sec:convergence}

Consider again the i.i.d.\ model: let $\X_n=\{X_1,\dots,X_n\}$ be i.i.d.\ observations of density $\kappa$ with respect to the $d$-dimensional Hausdorff measure on some $d$-dimensional manifold $\X$. The general question we are addressing in this section is the convergence of the observed diagrams $\dgmmap_q[\KK(\X_n)]$ to $\dgmmap_q[\KK(\X)]$, with $\KK$ either the Rips or the \v Cech filtration. Of course, the question has already been answered in some sense. For instance, Theorem \ref{thm:wass_stab} affirms that the sequence of observed diagrams will always converge to $\dgmmap_q[\KK(\X)]$ for the bottleneck distance,  if $\KK$ is the \v Cech filtration\footnote{A similar result states that the bottleneck distance between two diagrams, each built with the Rips filtration on some space, is controlled by the Hausdorff distance between the two spaces (see Theorem 3.1 \cite{chazal2009gromov}). As the Rips filtration, contrary to the \v Cech filtration, cannot be seen as the filtration of the sublevel sets of some function, this stability is not a consequence of Theorem \ref{thm:wass_stab}.}. However, this is not informative with respect to the convergence of the representations introduced in the previous section, which is related to a weak convergence of measure: For which functions $\phi$ does $\dgmmap_q[\KK(\X_n)](\phi)$ converge to $\dgmmap_q[\KK(\X)](\phi)$? 

The stability theorem for the bottleneck distance asserts that, for $\varepsilon >0$ small enough, and for $n$ large enough, $\dgmmap_q[\KK(\X_n)]$ can be decomposed into two separate sets of points: a set of fixed size $\dgmmap_{\text{true},q}[\KK(\X_n)]$ that is $\varepsilon$-close to points in $\dgmmap_q[\KK(\X)]$ and the remaining part of the diagram, $\dgmmap_{\text{noise},q}[\KK(\X_n)]$, usually consisting of a large number of points, which have persistence smaller than $\varepsilon$, i.e.\ these are the points that lie close to the diagonal. A Taylor expansion of $\phi$ shows that the difference between $\dgmmap_q[\KK(\X_n)](\phi)$ and $\dgmmap_q[\KK(\X)](\phi)$ is of the order of $\dgmmap_{\text{noise},q}[\KK(\X_n)](\pers^\alpha)$ for some $\alpha\geq 0$. The latter quantities are therefore of utmost interest to achieve our goal. Instead of directly studying $\dgmmap_{\text{noise},q}[\KK(\X_n)](\pers^\alpha)$ for $\X_n$ on a $d$-dimensional manifold, we focus on the study of the quantity $\dgmmap_{q}[\KK(\X_n)](\pers^\alpha)$ for $\X_n$ in a cube $[0,1]^d$.

Contributions to the study of quantities of the form $\dgmmap_q[\KK(\mathbb{S}_n)](\phi)$ have been made in  \cite{hiraoka2018limit}, where $\mathbb{S}_n$ is considered to be the restriction of a stationary process to a box of volume $n$ in $\R^d$. Specifically, \cite{hiraoka2018limit} shows the vague convergence of the rescaled diagram $n^{-1}\dgmmap_q[\KK(\mathbb{S}_n)]$ to some Radon measure $\mu_q$. The two recent papers \cite{trinh2017remark,goel2018strong} prove that a similar convergence actually holds for $\mathbb{S}_n$ a binomial sample on a manifold. However, vague convergence deals with continuous functions $\phi$ with compact support, whereas we are interested in functions of the type $\pers^\alpha$, which are not even bounded. Our contributions to the matter consists in proving, for samples on the cube $[0,1]^d$, a stronger convergence, allowing test functions to have non-compact support and polynomial growth. As a gentle introduction to the formalism used later, we first recall some known results from geometric probability on the study of Betti numbers, and we also detail relevant results of \cite{hiraoka2018limit, trinh2017remark, goel2018strong}.

\subsection{Prior work}
In the following, $\KK$ refers to either the \v Cech or the Rips filtration. Let $\kappa$ be a density on $[0,1]^d$ such that:
 \begin{equation}\label{eq:kappa}
0 <  \inf \kappa \leq  \sup\kappa < \infty.
\end{equation}
Note that the cube $[0,1]^d$ could be replaced by any compact convex body (i.e.\ the boundary of an open bounded convex set). However, the proofs (especially geometric arguments of Section \ref{sec:proof_bound}) become much more involved in this greater generality. To keep the main ideas clear, we therefore restrict ourselves to the case of the cube. We indicate, however, when challenges arise in the more general setting.

Let $(X_i)_{i\geq 1}$ be a sequence of i.i.d.\ random variables sampled from density $\kappa$ and let $(N_i)_{i\geq 1}$ be an independent sequence of Poisson variables with parameter $i$. In the following $\X_n$ denotes either $\{X_1,\dots,X_n\}$, a binomial process of intensity $\kappa$ and of size $n$, or $\{X_1,\dots,X_{N_n}\}$, a Poisson process of intensity $n\kappa$. The fact that the binomial and Poisson processes are built in this fashion is not important for weak laws of large numbers (only the law of the variables is of interest), but it is crucial for strong laws of large numbers to make sense.

The persistent Betti numbers $\beta^{r,s}(\dgmmap_q[\KK]) = \dgmmap_q[\KK(\X_n)](\ones_{[r,\infty)\times [s,\infty)]})$ are denoted more succinctly by $\beta^{r,s}_q(\KK)$. When $r=s$, we use the notation $\beta^r_q(\KK)$.

\begin{theorem}[Theorem 1.4 in \cite{trinh2017remark}]\label{thm:betti_strong}
Let $r>0$ and $q\geq 0$. Then, with probability one, $n^{-1}\beta_q^r(\KK(n^{1/d}\X_n))$ converges to some constant. The convergence also holds in expectation.
\end{theorem}

The theorem is originally stated with the \v{C}ech filtration but its generalization to the Rips filtration (or even to more general filtrations considered in \cite{hiraoka2018limit}) is straightforward. The proof of this theorem is based on a simple, yet useful geometric lemma, which still holds for the persistent Betti numbers, as proven in \cite{hiraoka2018limit}. Recall that for $j\geq 0$, $K_j$ denotes the $j$-skeleton of the simplicial complex $K$.

\begin{lemma}[Lemma 2.11 in \cite{hiraoka2018limit}]\label{lem:geomineq} Let $\X\subset \mathbb{Y}$ be two subsets of $\R^d$. Then 
\begin{equation}
|\beta_q^{r,s}(\KK(\X))- \beta_q^{r,s}(\KK(\mathbb{Y}))| \leq \sum_{j=q}^{q+1} |K_j^s(\mathbb{Y})\backslash K_j^s(\X)|.
\end{equation}
\end{lemma}

In \cite{hiraoka2018limit}, this lemma was used to prove the convergence of expectations of diagrams of stationary point processes. As indicated in \cite[Remark 2.4]{goel2018strong}, this lemma can also be used to prove the convergence of the expectations of diagrams for non-homogeneous binomial processes on manifold.
Let $C_c(\upperdiag)$ be the set of functions $\phi:\upperdiag \to \R$ with compact support. We say that a sequence $(\mu_n)_{n\geq 0}$ of measures on $\upperdiag$ \emph{converges $C_c$-vaguely} to $\mu$ if $\forall \phi\in C_c(\upperdiag)$, $\mu_n(\phi) \xrightarrow[n\to \infty]{} \mu(\phi)$. Note that this does not include the function $\phi = 1$ or the function $\phi=\pers$. Vague convergence is denoted by $ \xrightarrow[]{v_c}$. Set $\mu_n = n^{-1}\dgmmap[\KK(n^{1/d}\X_n)]$. Remark 2.4 in \cite{goel2018strong} implies the following theorem.

\begin{theorem}[Remark 2.4 in \cite{goel2018strong} and Theorem 1.5 in \cite{hiraoka2018limit}]\label{thm:duy}
Let $\kappa$ be a probability density function on a $d$-dimensional compact $C^1$ manifold $\X$, with $\int_{\X} \kappa^j(z)d z <\infty$ for $j\in \N$. Then, for $q\geq 0$, there exists a unique Radon measure $\mu^\kappa_q$ on $\upperdiag$ such that
\begin{equation}
E[\mu_n] \xrightarrow[n\to \infty]{v_c} \mu_q^\kappa
\end{equation}
and
\begin{equation}\label{eq:ergodic}
\mu_n \xrightarrow[n\to \infty]{v_c} \mu_q^\kappa \mbox{ a.s..}
\end{equation}
The measure $\mu_q^\kappa$ is called the persistence diagram of intensity $\kappa$ for the filtration $\KK$.
\end{theorem}


\subsection{Main results}

A function $\phi:\upperdiag \to \R$ is said to vanish on the diagonal if 
\begin{equation}
\lim_{\varepsilon\to 0} \sup_{\pers(\rD)\leq \varepsilon} |\phi(\rD)| = 0.
\end{equation}
Denote by $C_0(\upperdiag)$ the set of all such functions. The weight functions of Section \ref{sec:stability} all lie in $C_0(\upperdiag)$.  We say that a function $\phi:\upperdiag \to \R$ has polynomial growth if there exist two constants $A,\alpha>0$, such that
\begin{equation}\label{polyGrowth}
|\phi(\rD)| \leq A\left(1+ \pers(\rD)^\alpha\right).
\end{equation}

The class $C_{\mathrm{poly}}(\upperdiag)$ of functions in $C_0(\upperdiag)$ with polynomial growth constitutes a reasonable class of functions $w\cdot \phi$ one may want to build a representation with. Our goal is to extend the convergence of Theorem \ref{thm:duy} to this larger class of functions. Convergence of measures $\mu_n$ to $\mu$ with respect to $C_{\mathrm{poly}}(\upperdiag)$, i.e.\ $\forall \phi\in C_{\mathrm{poly}}(\upperdiag)$, $\mu_n(\phi) \xrightarrow[n\to \infty]{} \mu(\phi)$,  is denoted by $ \xrightarrow[]{v_p}$. Note that this class of functions is standard: it is for instance known to characterize $p$-th Wasserstein convergence in optimal transport (see \cite[Theorem 6.9]{villani2008optimal}). 

\begin{theorem}\label{thm:main_cv}
(i) For $q\geq 0$, there exists a unique Radon measure $\mu_q^\kappa$ such that $E[\mu_q^n] \xrightarrow[n\to \infty]{v_p} \mu_q^\kappa$ and, with probability one, $\mu_q^n \xrightarrow[n\to \infty]{v_p} \mu_q^\kappa$. The measure $\mu_q^\kappa$ is called the $q$-th persistence diagram of intensity $\kappa$ for the filtration $\KK$. It does not depend on whether $\X_n$ is a Poisson or a binomial process, and is of positive finite mass.\\[5pt]
(ii) The convergence also holds pointwise for the $L_p$ distance: for all $\phi \in C_{\mathrm{poly}}(\upperdiag)$, and for all $p\geq 1$, $\mu_q^n(\phi) \xrightarrow[n\to \infty]{L_p} \mu_q^{\kappa}(\phi)$. In particular, $|\mu_q^\kappa(\phi)|<\infty$.
%
\end{theorem}

\begin{remark} {\em (a)} Remark 2.4 together with Theorem 1.1 in \cite{goel2018strong} imply that the measure $\mu_q^\kappa$ has the following expression:
\begin{equation}\label{expressionmu}
\mu_q^\kappa(\phi) = E\left[\int_{\upperdiag} \phi(\rD\kappa(X)^{-1/d})d\mu_q(\rD)\right] \ \forall \phi \in C_c(\upperdiag),
\end{equation}
where $\mu_q = \mu_q^{\ones}$ is the $q$-th persistence diagram of uniform density on $[0,1]^d$, appearing in Theorem \ref{thm:duy}, and the expectation is taken with respect to a random variable $X$ having a density $\kappa$. 

{\em (b)} Assume $q=0$ and $d=1$. Then, the persistence diagram $\dgmmap_0[\KK(\X_n)]$ is simply the collection of the intervals $(X_{(i+1)}-X_{(i)})$ where $X_{(1)}<\cdots <X_{(n)}$ is the order statistics of $\X_n$. The measure $E[\mu_0^n]$ can be explicitly computed: it converges to a measure having density $u \mapsto E[\exp(-u\kappa(X))\kappa(X)]$ with respect to the Lebesgue measure on $\R_+$, where $X$ has density $\kappa$. Take $\kappa$ the uniform density on $[0,1]$: one sees that this is coherent with the basic fact that the spacings of a homogeneous Poisson process on $\R$ are distributed according to an exponential distribution. Moreover, the expression \eqref{expressionmu} is found again in this special case.\\[5pt]
{\em (c)} Theorem 1.9 in \cite{hiraoka2018limit} states that the support of $\mu^{\ones}_q$ is $\upperdiag$. Using equation \eqref{expressionmu}, the same holds for $\mu^{\kappa}_{q}$.\\[5pt]
{\em (d)} Theorem \ref{thm:main_thm} is a direct corollary of Theorem \ref{thm:main_cv}. Indeed, we have
\begin{align*}
n^{\frac{\alpha}{d}-1}  \Pers_\alpha(\dgmmap_q[\KK(\X_n)]) &\defeq n^{\frac{\alpha}{d}-1} \sum_{\rD \in\dgmmap_q[\KK(\X_n)]} \pers^\alpha(\rD) \\
&= n^{-1} \sum_{\rD \in\dgmmap_q[\KK(\X_n)]} \pers^\alpha(n^{-\frac{1}{d}}\rD) \\
&= \mu_q^n(\pers^\alpha),
\end{align*}

a quantity which converges to $\mu_q^\kappa(\pers^\alpha)$. The relevance of Theorem \ref{thm:main_thm} is illustrated in Figure \ref{fig:experiments}, where \v Cech complexes are computed on random samples on the torus.
\end{remark}

The core of the proof of Theorem \ref{thm:main_cv} consists in a control of the number of points appearing in diagrams. This bound is obtained thanks to geometric properties satisfied by the \v Cech and Rips filtrations. Finding good requirements to impose on a filtration $\KK$ for this control to hold is an interesting question. The following states some non-asymptotic controls of the number of points in diagrams which are interesting by themselves.
\begin{figure}
\centering
\begin{tabular}{c| >{\centering\arraybackslash}m{5cm} | >{\centering\arraybackslash}m{5cm} }
Weight function & $n=500$ & $n=2000$ \\
\hline
$\pers^0$ & \includegraphics[width=50mm]{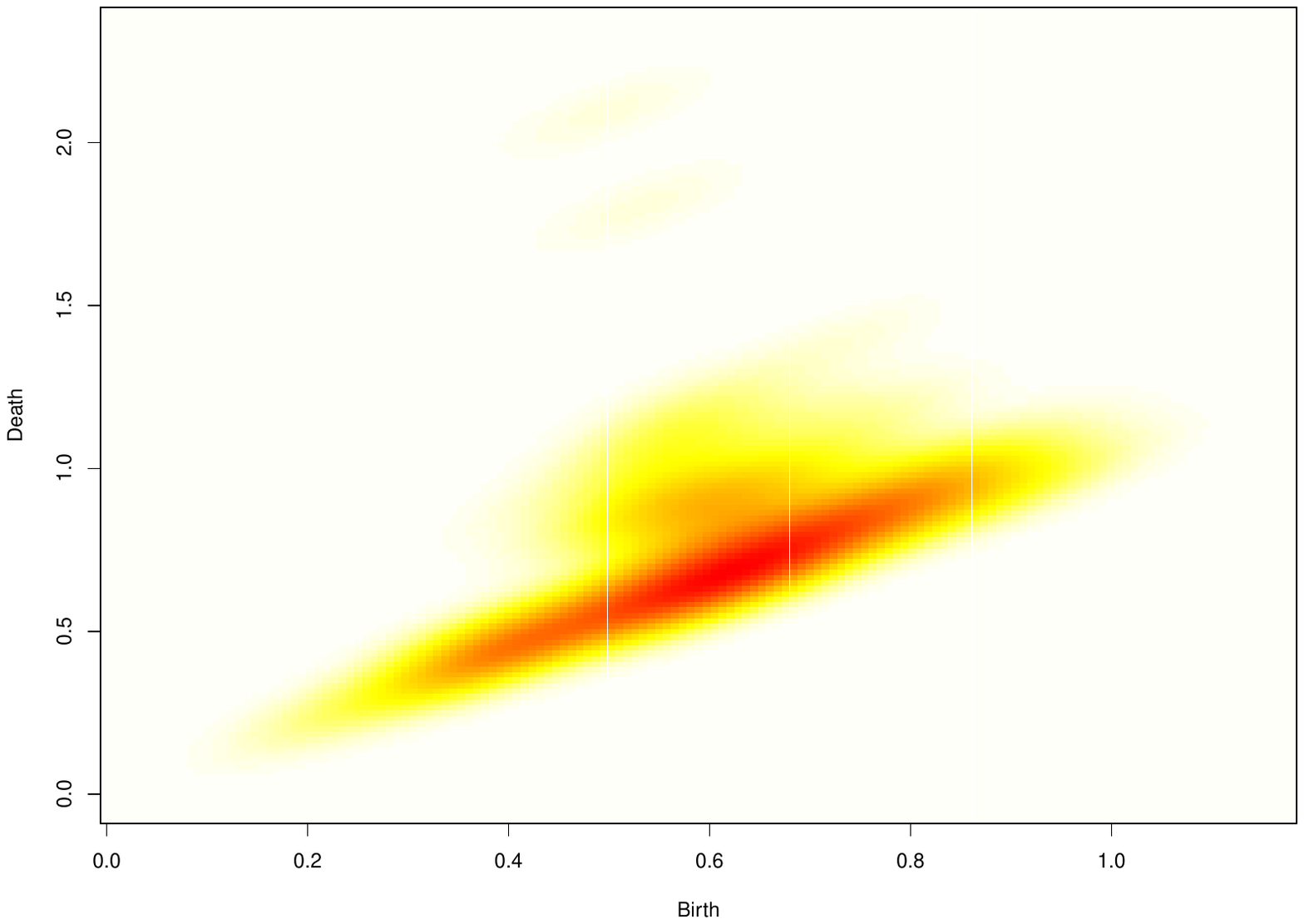} & \includegraphics[width=50mm]{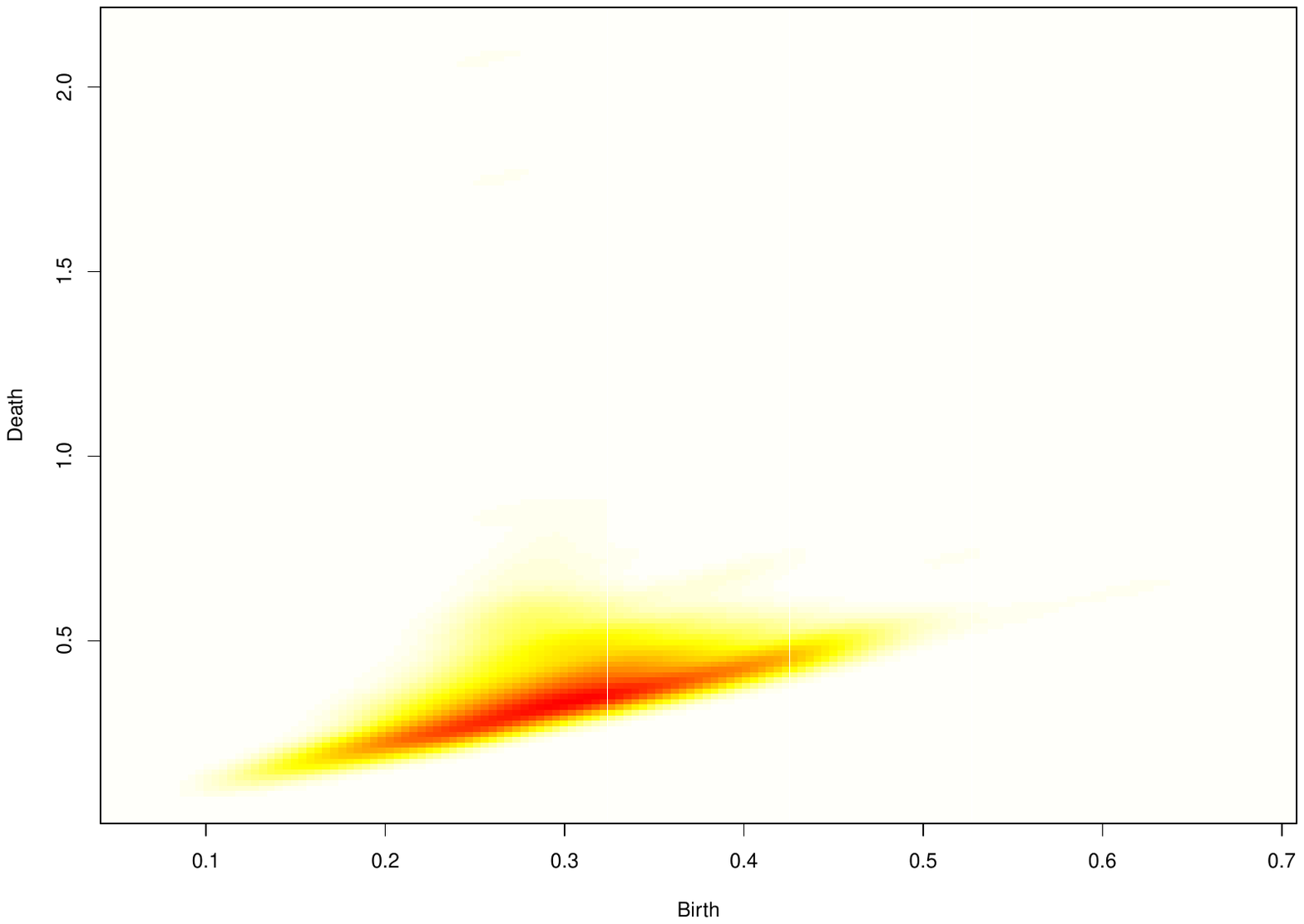} \\
\hline
$\pers^1$  & \includegraphics[width=50mm]{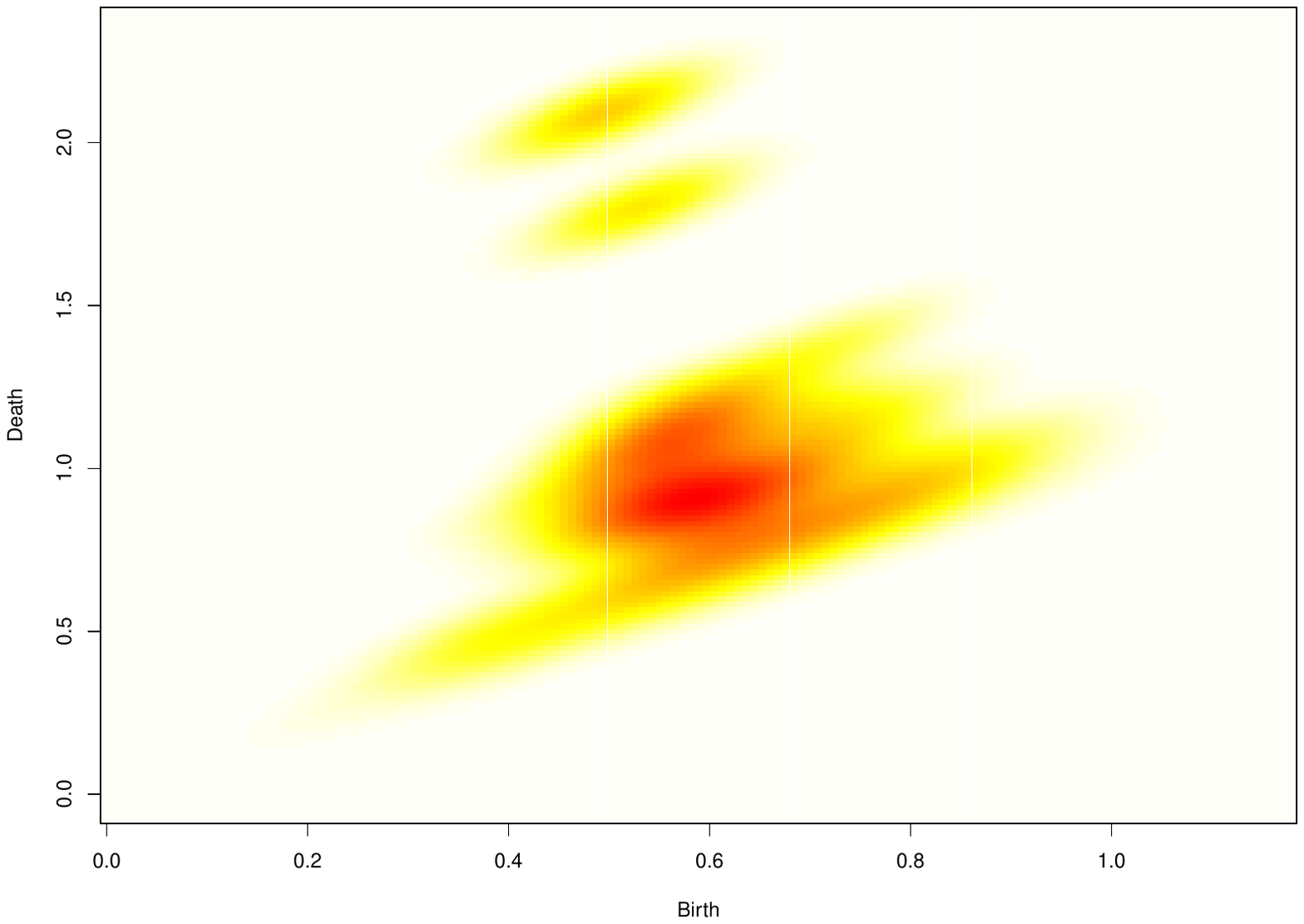} & \includegraphics[width=50mm]{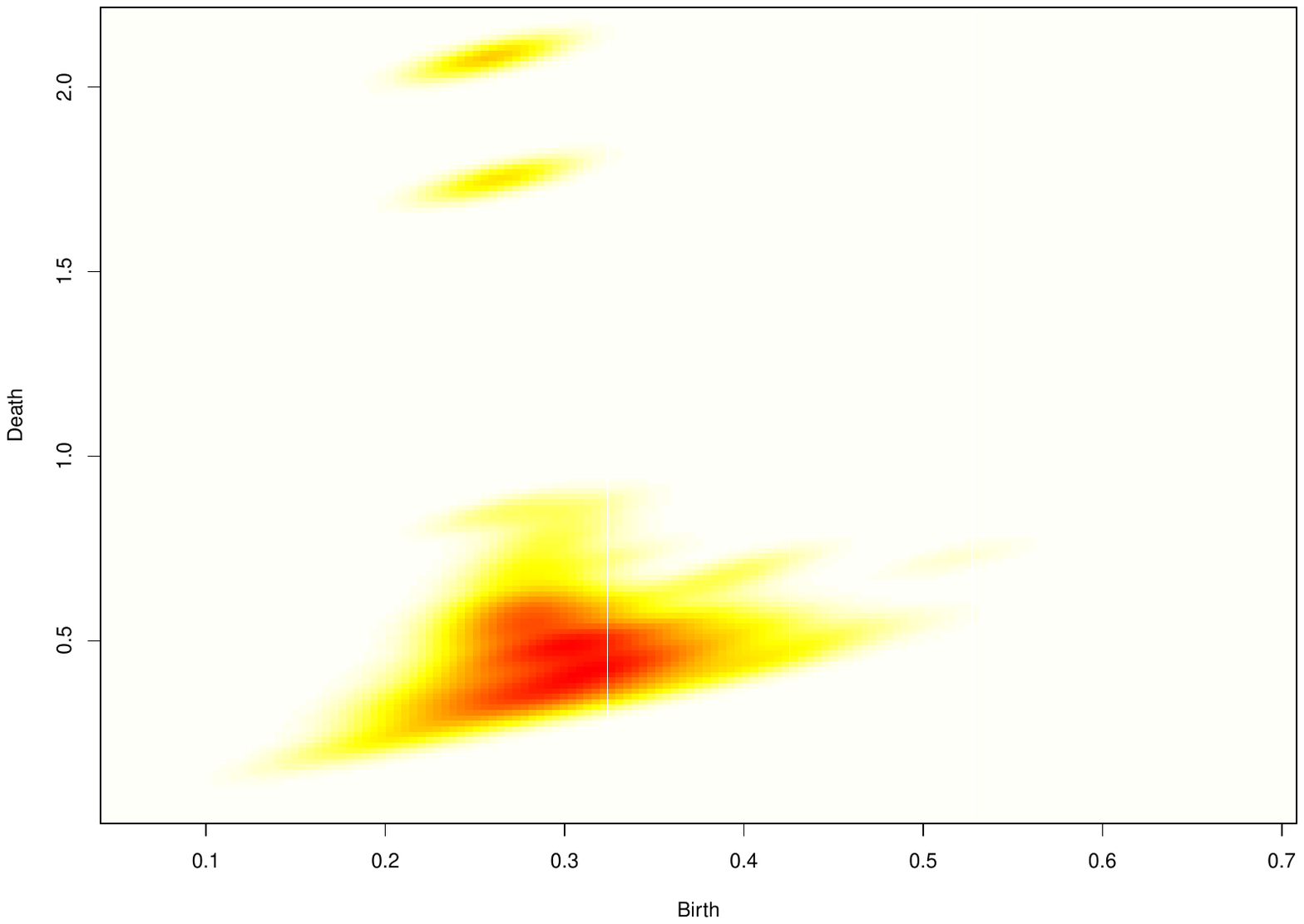} \\
\hline
$\pers^2$  & \includegraphics[width=50mm]{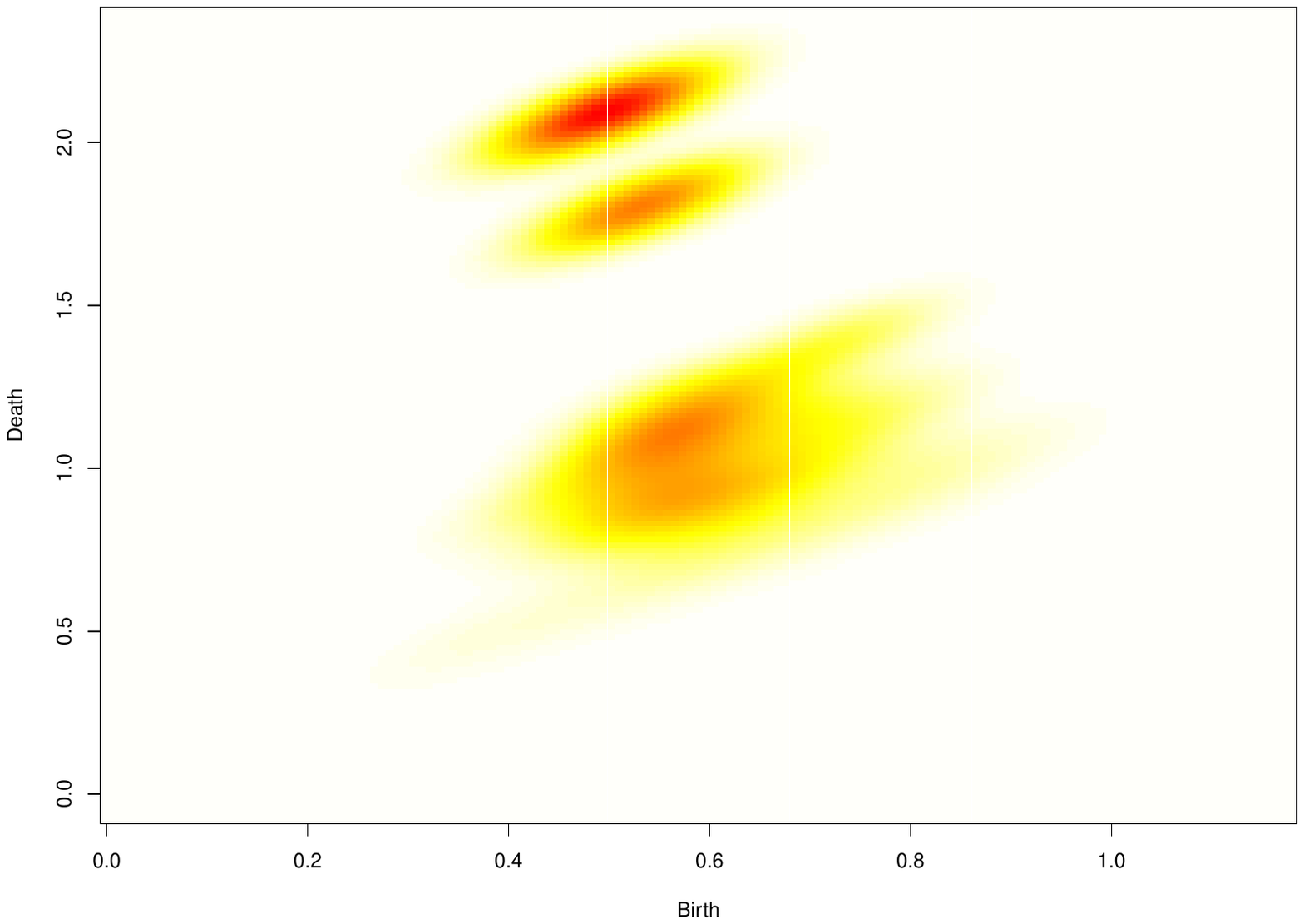} & \includegraphics[width=50mm]{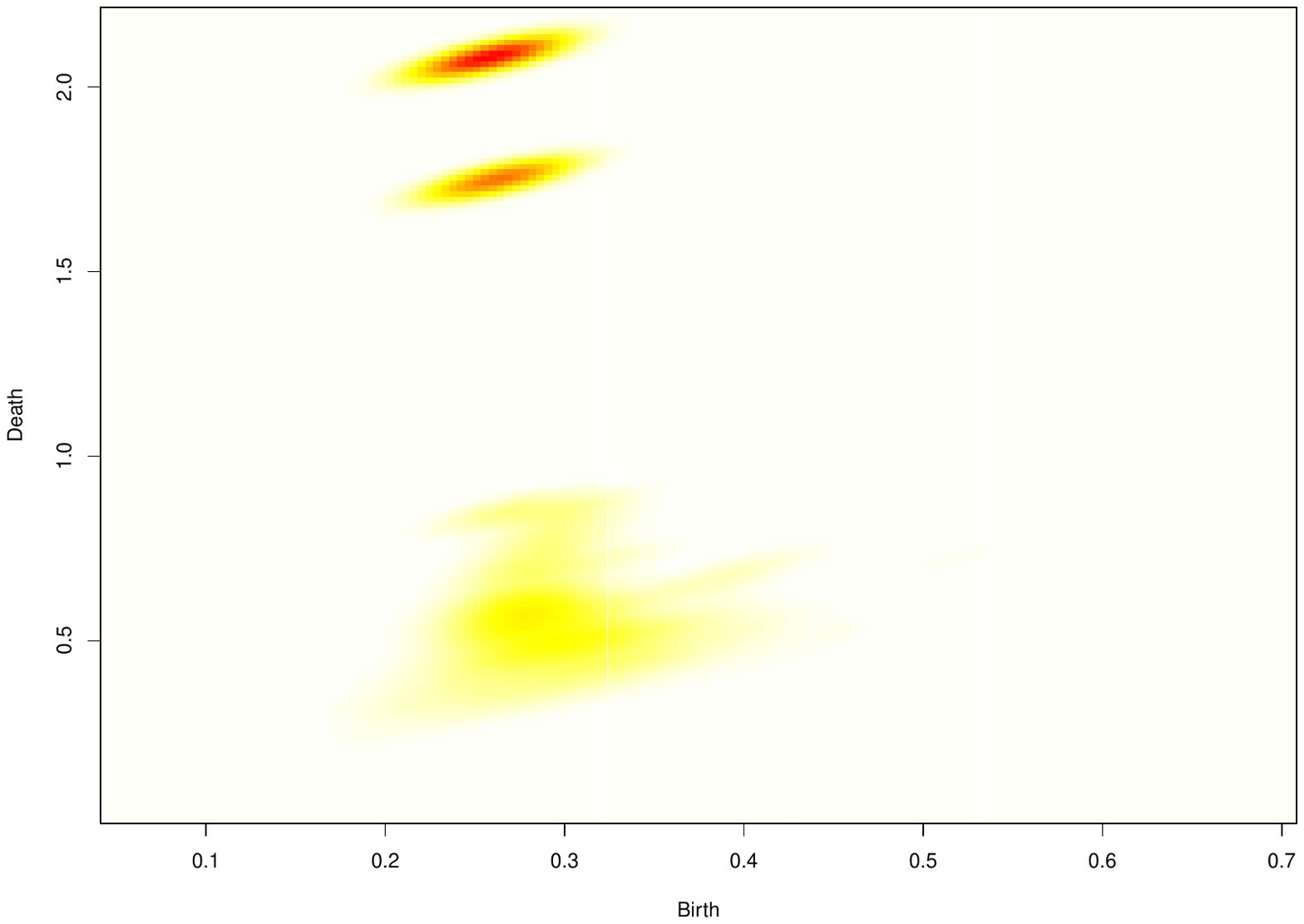} \\
\hline
$\pers^{100}$ & \includegraphics[width=50mm]{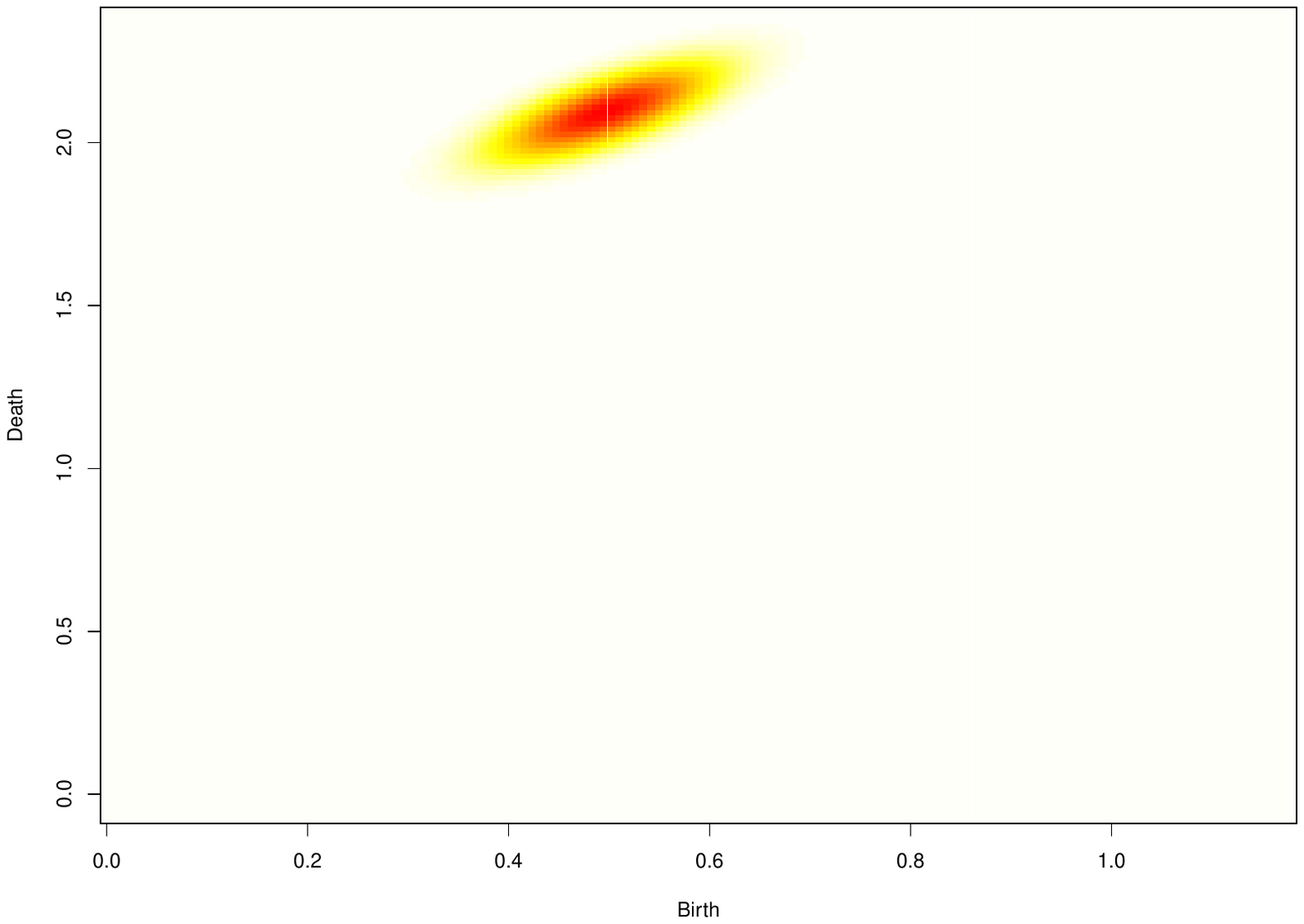} & \includegraphics[width=50mm]{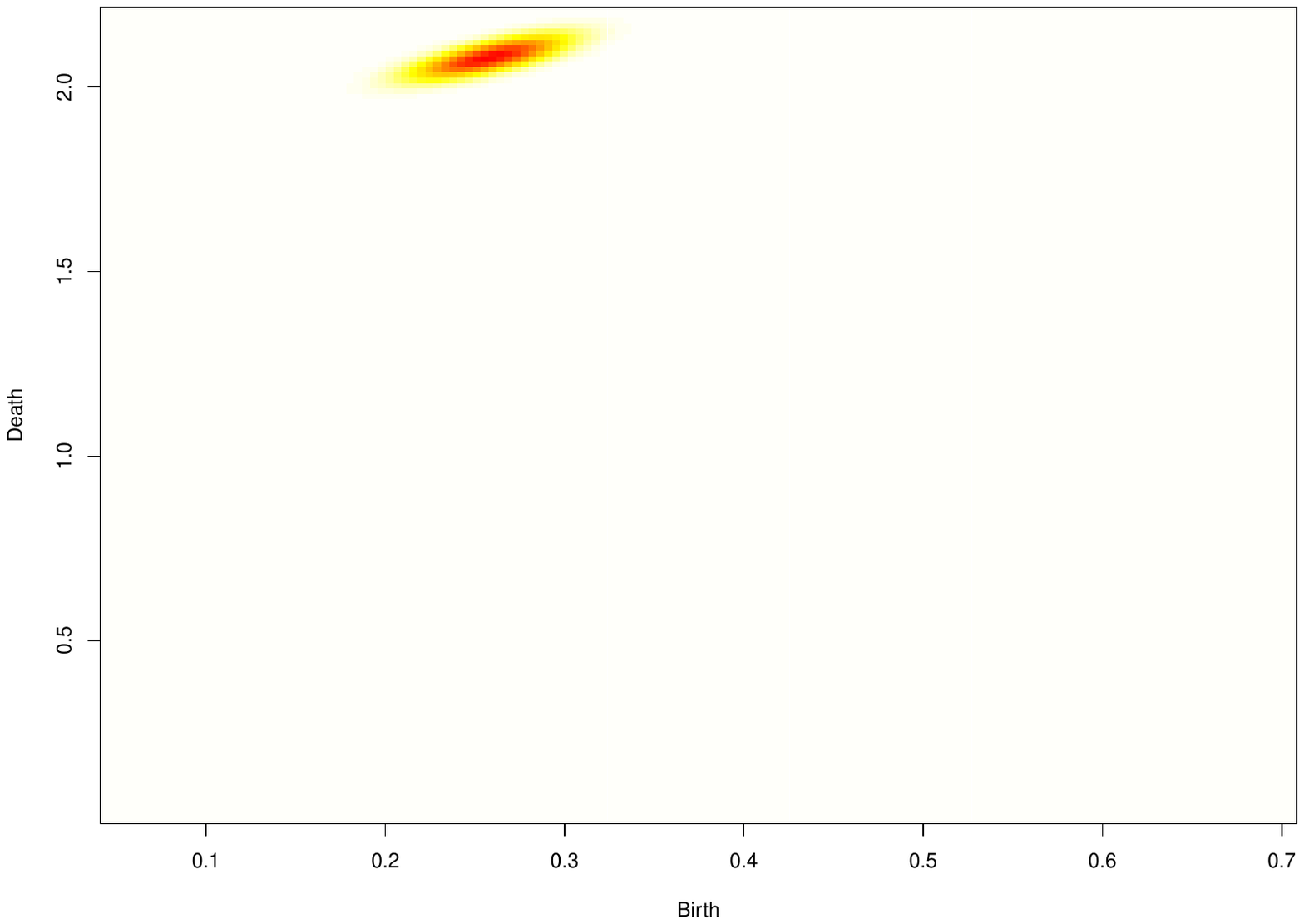} \\
\end{tabular}
\caption{For $n=500$ or $2000$ points uniformly sampled on the torus, persistence images \cite{adams2017persistence} for different weight functions are displayed. For $\alpha <2$, the mass of the topological noise is far larger than the mass of the true signal, the latter being comprised by the two points with high-persistence. For $\alpha=2$, the two points with high-persistence are clearly distinguishable. For $\alpha =100$, the noise has also disappeared, but so has one of the point with high-persistence.}
\label{fig:experiments}
\end{figure}

\begin{prop}\label{prop:bound_size_diagrams} Let $M\geq 0$ and define $U_M = \R \times [M,\infty)$. Then, there exists constants $c_1,c_2>0$ (which can be made explicit) depending on $\kappa$ and $q,$ such that, for any $t>0$,
\begin{equation}\label{eq:tail_bound}
P(\mu_q^n(U_M)>t) \leq c_1\exp(-c_2(M^d+t^{1/(q+1)})).
\end{equation}
\end{prop}
As an immediate corollary, the moments of the total mass $|\mu_n|$ are uniformly bounded. However, the proof of the almost sure finiteness of $\sup_n |\mu_n|$ is much more intricate. Indeed, we are unable to control directly this quantity, and we prove that a majorant of $|\mu_n|$ satisfies concentration inequalities. The majorant arises as the number of simplicial complexes of a simpler process, whose expectation is also controlled.

It is natural to wonder whether $\mu_q^\kappa$ has some density with respect to the Lebesgue measure on $\upperdiag$: it is the case for the for $d=1$, and it is shown in \cite{chazal_et_al:LIPIcs:2018:8739} that $E[\mu_q^n]$ also has a density. Even if those elements are promising, it is not clear whether the limit $\mu_q^\kappa$ has a density in a general setting. However, we are able to prove that the marginals of $\mu_q^\kappa$ have densities.
\begin{prop}\label{prop:density} Let $\pi_1$ (resp. $\pi_2$) be the projection on the $x$-axis (resp. $y$-axis). Then,  for $q>0$, the pushforwards $\pi_1^\star (\mu_q^\kappa)$ and $\pi_2^\star (\mu_q^\kappa)$ have densities with respect to the Lebesgue measure on $\R$. For $q=0$, $\pi_2^\star (\mu_q^\kappa)$ has a density.
\end{prop}

\section{Discussion}\label{sec:discussion}
The tuning of the weight functions in the representations of persistence diagrams is a critical issue in practice. When the statistician has good reasons to believe that the data lies near a $d$-dimensional structure, we give, through two different approaches, an heuristic to tune this weight function: a weight of the form $\pers^\alpha$ with $\alpha\geq d$ is sensible. The study carried out in this paper allowed us to show new results on the asymptotic structure of random persistence diagrams. While the existence of a limiting measure in a weak sense was already known, we strengthen the convergence, allowing a much larger class of test functions. Some results about the properties of the limit are also shown, namely that it has a finite mass, finite moments, and that its marginals have densities with respect to the Lebesgue measure. Challenging open questions include:
\begin{itemize}
	\item Convergence of the rescaled diagrams $\mu_n$ with respect to some transport metric: The main issue consists in showing that one can extend, in a meaningful way, the distance $W_p$ to general Radon measures.  This is the topic of a recent work (see Section 5.1 in \cite{divol2019understanding}).
	\item Existence of a density for the limiting measure: An approach for obtaining such results would be to control the numbers of points of a diagram in some square $[r_1,r_2]\times [s_1,s_2]$.
	\item Convergence of the number of points in the diagrams: The number of points in the diagrams is a quantity known to be not stable (motivating the use of bottleneck distances, which is blind to them). However, experiments show that this number, conveniently rescaled, converges in this setting. An analog of Lemma \ref{lem:geomineq} for the number of points in the diagrams with small persistence would be crucial to attack this problem.
	\item Generalization to manifolds: While the vague convergence of the rescaled diagrams is already proven in \cite{goel2018strong}, allowing test functions without compact support seems to be a challenge. Once again, the crucial issue consists in controlling the total number of points in the diagrams.
	\item Dimension estimation: We have proved that the total persistence of a diagram built on a given point cloud depends crucially on the intrinsic dimension of such a point cloud. Inferring the dependence of the total persistence with respect to the size of the point cloud (through subsampling) leads to estimators of this intrinsic dimension. Studying the properties of such estimators is the topic of an on-going work of Henry Adams and co-authors (personal communication).
\end{itemize}

\section{Proofs}\label{sec:proofs}

\subsection{Proof of Theorem \ref{thm:main_stab}}\label{subsec:proof_stab}
We only treat the case $p<\infty$, the proof being easily adapted to the case $p=\infty$. Introduce for $\mu,\nu$ two measures of mass  $m>0$ on  $\upperdiag$, the Monge-Kantorovitch distance between $\mu$ and $\nu$:
\begin{equation}
    d_{\mathrm{MK}}(\mu,\nu) \defeq \sup\left\{ \mu(\phi)-\nu(\phi),\ \phi:\upperdiag \to \R\ \text{1-Lipschitz}\right\}.
\end{equation}

Fix two persistence diagrams $\dgm_1$ and $\dgm_2$. Denote $\mu = \dgm_1(w \ \cdot)$ (resp. $\nu = \dgm_2(w \ \cdot)$) the measure having density $w$ with respect to $\dgm_1$ (resp. $\dgm_2$). For $\gamma$ a matching attaining the $p$-th Wasserstein distance between $\dgm_1$ and $\dgm_2$, denote $\tilde{\mu} =\sum_{\rD\in \dgm_1\cup \Delta}w(\gamma(\rD))\delta_{\rD}$.
We have
\begin{align}
\|\Phi_w(\dgm_1)-\Phi_w(\dgm_2)\|_{\mathcal{B}} &= \|\mu(\phi)-\nu(\phi)\|_\mathcal{B}\leq \|\mu(\phi)-\tilde{\mu}(\phi)\|_\mathcal{B} + \|\tilde{\mu}(\phi) - \nu(\phi)\|_\mathcal{B} \nonumber \\
&\leq \|\phi\|_\infty |\mu - \tilde{\mu}| + \Lip(\phi) d_{\mathrm{MK}}(\tilde{\mu},\nu).\label{eq:stab1}
\end{align}
We bound the two terms in the sum separately. Let us first bound $d_{\mathrm{MK}}(\tilde{\mu},\nu)$. The Monge-Kantorovitch distance is also the infimum of the costs of transport plans between $\tilde{\mu}$ and $\nu$ (see \cite[Chapter 2]{villani2008optimal} for details), so that
\[d_{\mathrm{MK}}(\tilde{\mu},\nu)\leq  \sum_{\rD\in \dgm_1 \cup \Delta} w(\gamma(\rD))\|r-\gamma(r)\|.\]
Define $q$ such that $\frac{1}{p}+\frac{1}{q}=1$. As condition \eqref{eq:grad_weight} implies that $\|w(\rD)\| \leq \frac{A}{\alpha} \pers(\rD)^{\alpha}$, the distance $d_{\mathrm{MK}}(\tilde{\mu},\nu)$ is bounded by
\begin{align}
\sum_{\rD\in \dgm_1 \cup \Delta} w(\gamma(\rD))\|r-\gamma(r)\| &\leq \left(\sum_{\rD\in \dgm_1 \cup \Delta} w(\gamma(\rD))^q\right)^{1/q} \left(\sum_{\rD\in \dgm_1 \cup \Delta} \|r-\gamma(r)\|^p \right)^{1/p} \nonumber \\
&\leq  \frac{A}{\alpha}\left(\sum_{\rD\in \dgm_1 \cup \Delta} \pers(\gamma(\rD))^{q\alpha}\right)^{1/q} W_p(\dgm_1,\dgm_2) \nonumber \\ 
&\leq \frac{A}{\alpha}\left(G\left\{q\alpha\right\} \right)^{1/q} W_p(\dgm_1,\dgm_2). \label{eq:stab2}
\end{align}
We now treat the first part of the sum in \eqref{eq:stab1}. For $\rD_1, \rD_2$, in $\upperdiag$ with $\pers(\rD_1) \leq \pers(\rD_2)$, define the path with unit speed $h : [\pers(\rD_1),\pers(\rD_2)] \to \upperdiag$ by 
\[ h(t) = \rD_2 \frac{t- \pers(\rD_1)}{\pers(\rD_2)-\pers(\rD_1)} + \rD_1 \frac{\pers(\rD_2)-t}{\pers(\rD_2)-\pers(\rD_1)},\]
so that it satisfies $\pers(h(t)) = t$. The quantity $|w(\rD_1) - w(\rD_2)|$ is bounded by
\begin{align*}
\int_{\pers(\rD_1)}^{\pers(\rD_2)} |\nabla w(h(t)).h'(t)|dt &\leq \int_{\pers(\rD_1)}^{\pers(\rD_2)} A\ \pers(h(t))^{\alpha-1}dt \\
&\leq \int_{\pers(\rD_1)}^{\pers(\rD_2)} A\ t^{\alpha-1}dt \\
&= \frac{A}{\alpha}(\pers(\rD_2)^\alpha-\pers(\rD_1)^\alpha).
\end{align*}
For $0<y<x$ and $0\leq a \leq 1$, using the convexity of $t\mapsto t^\alpha$, it is easy to see that $x^\alpha - y^\alpha \leq \alpha(x-y)^a x^{\alpha-a}$. Define $p' = \frac{p}{a}$, $q' = \frac{p'}{p'-1}$ and  $M(\rD) \defeq \max(\pers(\rD),\pers(\gamma(\rD)))$. We have,
\begin{align}
|\mu-\tilde{\mu}| &= \sum_{\rD\in \dgm_1 \cup \Delta} |w(\rD)-w(\gamma(\rD))| \nonumber \\
&\leq A \sum_{\rD \in \dgm_1 \cup \Delta}|\pers(\rD)-\pers(\gamma(\rD))|^a M(\rD)^{\alpha-a} \nonumber \\
&\leq  A\left(\sum_{\rD \in \dgm_1 \cup \Delta}|\pers(\rD) - \pers(\gamma(\rD))|^{ap'}\right)^{1/p'}   \left(\sum_{\rD \in \dgm_1 \cup \Delta}  M(\rD)^{q'(\alpha-a)} \right)^{1/q'} \nonumber \\
&\leq  A W_{ap'}(\dgm_1,\dgm_2)^{a}   \left(\sum_{\rD \in \dgm_1 \cup \Delta}  \left(\pers(\rD)^{q'(\alpha-a)}+\pers(\gamma(\rD))^{q'(\alpha-a)}\right) \right)^{1/q'} \nonumber \\
&\leq  A W_{ap'}(\dgm_1,\dgm_2)^{a}  2^{1/q'} \left(G\{q'(\alpha-a)\} \right)^{1/q'} \label{eq:stab3} 
\end{align}
Combining equations \eqref{eq:stab1}, \eqref{eq:stab2} and \eqref{eq:stab3} concludes the proof.

\subsection{Proof of Corollary \ref{cor:stab_bis}}
Corollary \ref{cor:stab_bis} follows easily by using the definition of a space with bounded $m$-th total persistence along with the inequality $G\{t_1+t_2\} \leq \ell^{t_1}G\{t_2\}$. 

\subsection{Proof of Corollary \ref{cor:first_asymptotics}}
As already discussed, Theorem \ref{thm:main_stab} can be applied with $f_n = d(\cdot,\X_n)$ and $f$ the null function on the manifold $\X$.  Take $p=\infty$, $d<\alpha$ and $0<a< \min(1,\alpha-d)$:
\begin{equation}\label{proofCorStab}
\begin{aligned}
\|\Phi_w(\dgmmap_q[f_n])-\Phi_w(\dgmmap_q[f])\|_{\mathcal{B}} &\leq \Lip(\phi)\frac{A}{\alpha}G\{\alpha\} W_\infty(\dgmmap_q[f_n],\dgmmap_q[f]) \\
 &\hspace*{-1.5cm}+2\|\phi\|_\infty A G\{\alpha-a\} W_\infty(\dgmmap_q[f_n],\dgmmap_q[f])^a.
\end{aligned}
\end{equation}
 It is mentioned at the end of Section 2 in \cite{cohen2010lipschitz} that, for $m>d$, $\Pers_m(\dgmmap_q[f_n]) \leq mC_{\X} \|f_n\|_\infty^{m-d}/(m-d)$ for some constant $C_{\X}$ depending only on $\X$. Moreover, the stability theorem for the bottleneck distance ensures that $W_\infty(\dgmmap_q[f_n],\dgmmap_q[f]) \leq \|f_n\|_\infty$. Therefore,
 \begin{align}
\|\Phi_w(\dgmmap_q[f_n])-\Phi_w(\dgmmap_q[f])\|_{\mathcal{B}}
&\leq \Lip(\phi)\frac{\alpha AC_{\X}}{\alpha-d} \|f_n\|_\infty^{\alpha-d +1}  + \|\phi\|_\infty \frac{2AC_{\X}(\alpha-a)}{\alpha-a-d} \|f_n\|_\infty^{\alpha-a-d+a} \nonumber \\
 &\leq \Lip(\phi)\frac{\alpha AC_{\X}}{\alpha-d} \|f_n\|_\infty^{\alpha-d +1}  + \|\phi\|_\infty \frac{2 A C_{\X}\alpha}{\alpha-d} \|f_n\|_\infty^{\alpha-d}, \label{eq:cor_sum}
 \end{align}
 where, in the last line, the second term was minimized over $a$. The quantity $\|f_n\|_\infty$ is the Hausdorff distance between $\X_n$ and $\X$. Elementary techniques of geometric probability (see for instance \cite{cuevas2009set}) show that if $\X$ is a compact $d$-dimensional manifold, then $E[\|f_n\|_\infty^\beta] \leq c \left( \frac{\ln n}{n}\right)^{\beta /d}$ for $\beta \geq 0$, where $c$ is some constant depending on $\beta, \X, \inf \kappa$ and $\sup \kappa$. Therefore, the first term of the sum \eqref{eq:cor_sum} being negligible,
\begin{align*}
E\left[ \|\Phi_w(\dgmmap_q[f_n]) - \Phi_w(\dgmmap_q[f]) \|_{\mathcal{B}}\right] 
& \leq \|\phi\|_\infty \frac{2 AC_{\X}\alpha}{\alpha-d} c \left( \frac{\ln n}{n}\right)^{(\alpha-d) /d} + o\left(\left( \frac{\ln n}{n}\right)^{(\alpha-d) /d}\right).
\end{align*}
In particular, the conclusion holds for any $C >  2 AC_{\X} c $, for $n$ large enough.\hfill $\square$

\paragraph{} We now prove the propositions of Section \ref{sec:convergence}. In the following proofs, $c$ is a constant, depending on $\kappa$, $d$ and $q,$ which can change from line to line (or even represent two different constants in the same line). A careful read can make all those constants explicit. If a constant depends also on some additional parameter $x$, it is then denoted by $c(x)$.

\subsection{Proof of Proposition \ref{prop:bound_size_diagrams}}\label{sec:proof_bound}
First, as the right hand side of the inequality \eqref{eq:tail_bound} does not depend on $n$, one may safely assume that $\mu_n$ is built with the binomial process. The proof is based on two observations.\\[5pt]
(i) Let $r(\sigma) \defeq \min \{r>0, \sigma \in K^r(\X_n)\}$ denote the \emph{filtration time} of $\sigma$. A simplex $\sigma$ is said to be \emph{negative} in the filtration $\KK(\X_n)$ if $\sigma$ is not included in any cycle of $K^{r(\sigma)}(\X_n)$. A basic result of persistent homology states that points in $\dgmmap_q[\KK(\X_n)]$ are in bijection with pair of simplexes, one negative and one positive (i.e.\ non-negative). Moreover, the death time $r_2$ of a point $\rD=(r_1,r_2)$ of the diagram is exactly $r(\sigma)$ for some negative $(q+1)$-simplex $\sigma$. Therefore, $n\mu_n(U_M)$ is equal to $N_q(\X_n,M)$, the number of negative $(q+1)$-simplexes in the filtration $\KK(\X_n)$ appearing after $M_n \defeq n^{-1/d}M$. More details about this pairing between simplexes of the filtration can, for instance, be found in \cite{edelsbrunner2010computational}, section VIII.1.\\[5pt]
(ii) The number of negative simplexes in the \v Cech and Rips filtration can be efficiently bounded thanks to elementary geometric arguments.

\subsubsection{Geometric arguments for the Rips filtration}
We have
\begin{equation} N_q(\X_n,M) = \frac{1}{q+2} \sum_{i=1}^n \#\Xi(X_i,\X_n),\end{equation}
where, for $x\in \X$, with $\X$ a finite set, $\Xi(x,\X)$ is the set of negative $(q+1)$-simplexes (and therefore of size $q+2$) in $\RR(\X)$ that are containing $x,$ and have a filtration time larger than $M_n$.  The following construction is inspired by the proof of Lemma 2.4 in \cite{mcgivney1999asymptotics}. 

The angle (with respect to $0$) of two vectors $x,y \in \R^d$ is defined as

\[ \angle xy \defeq \arccos \left(\frac{\langle x,y \rangle}{\|x\| \|y\|}\right).\]
The angular section of a cone $A$ is defined as $\sup_{x,y \in A} \angle xy$. Denote by $C(x,r)$ the cube centered at $x$ of side length $2r$. For $0<\delta <1$, and for each face of the cube $C(x,r),$ consider a regular grid with spacing $\delta r$, so that the center of each face is one of the grid points. This results in a partition of the boundary of the cube $C(x,r)$ into $(d-1)$-dimensional cubes $\big(C_j^\delta(x,r)\big)_{j=1\dots Q}$ of side length $\delta r.$ Using this partition of the boundary of $C(x,r),$ we construct a partition of $C(x,r)$ into closed convex cones $(A_j^\delta(x,r))_{j=1\dots Q}$, where each cone $A_j^\delta(x,r)$ is defined as a $d$-simplex spanned by $x$ and one of the $(d-1)$-dimensional cubes $C_j^\delta(x,r)$ of side length $\delta r$ on a face of $C(x,r)$. In other words, the point $x$ is the apex of each $A_j^\delta(x,r)$, and $C_j^\delta(x,r)$ is its base. We call two such cones $A_j^\delta(x,r)$ and $A_{j^\prime}^\delta(x,r)$ {\it adjacent,} if $A_j^\delta(x,r) \cap A_{j^\prime}^\delta(x,r) \ne \{x\}.$

Fix $0<\eta <1,$ and define $R_{\delta,\eta}(x,\X_n)$ to be the smallest radius $r$ so that each cone $A_j^\delta(x,\eta r)$ in $C(x,\eta r)$ either contains a point of $\X_n$ other than $x$, or is not a subset of $(0,1)^d$ (see Figure \ref{fig:cones} for an illustration).

\begin{figure}
\centering
      \begin{tabular}{cc}
        \includegraphics[width=70 mm]{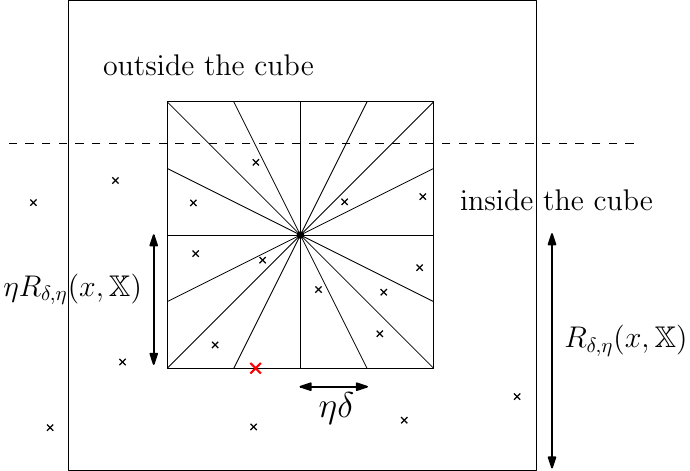}
    
    &
      
        \includegraphics[width=70 mm]{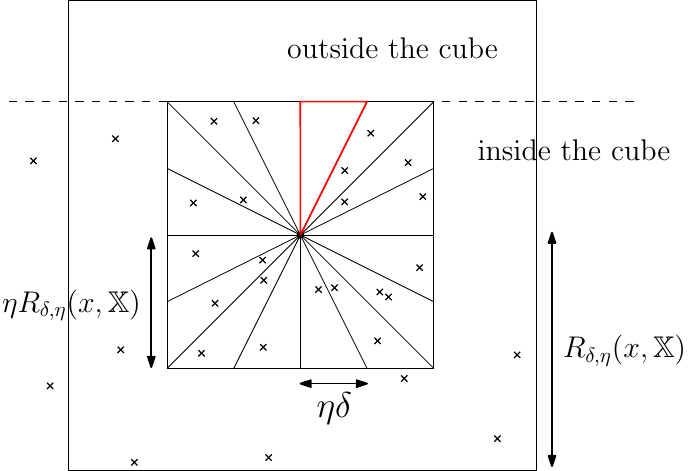}
    \end{tabular}
    \caption{\small \it Illustration of the definition of $R\defeq R_{\delta,\eta}(x,\X)$ for some two point clouds $\X$. The dashed line indicates the boundary of $[0,1]^d$. On the left display, the radius $R$ is such that there is a point (indicated in red) on $C_j^\delta(x,R)$. On the right display, there is a cone $A_j^\delta(x,R)$, indicated in red, for which $C_j^\delta(x,R)$ is on some face of the cube $[0,1]^d$.}
    \label{fig:cones}
\end{figure}

\begin{lemma}\label{lem:technical_geom} Let $x\in [0,1]^d$. Fix $\delta >0,$ and $0 < r \leq \frac{1}{2}$, and let $A_j^\delta(x, r)$ be a cone of $C(x,r)$ whose base $C_j^\delta(x,r)$ intersects $[0,1]^d$. Then, either $A_j^\delta(x, r)$ is a subset of $[0,1]^d$, or there exists a cone $A_{j'}^\delta(x,r)$ of $C(x,r)$ adjacent to $A_j^\delta(x, r)$ that is a subset of $[0,1]^d$.
\end{lemma}

{\em Proof.}  A necessary and sufficient condition for a cone $A_j^\delta(x,r)$ to be a subset of $[0,1]^d$ is that $C_j^\delta(x,r) \subset [0,1]^d$. Suppose that this is not the case, i.e.\ we have $C_j^\delta(x,r) \setminus [0,1]^d \ne \emptyset.$ For each coordinate $i=1,\dots, d$ for which $C_j^\delta(x,r)$ extends beyond a face of $[0,1]^d$, move one step in the `opposite' direction, and find the corresponding adjacent cone. The fact that $r\leq 1/2$ ensures that these (at most $d$) steps, each of size $r\delta,$ do not make the exterior boundary of the corresponding adjacent cone extend beyond any of the opposite faces of the cube corresponding to the directions of the steps.  \hfill $\square$

\vspace*{0.3cm}
Note that the angular section (with respect to $x$) of the union of a cone $A_j^\delta(x,r)$ and its adjacent cones is bounded by $c \delta$ for some constant $c$.

\begin{lemma}\label{lem:geom_lemma} Let $\eta = \min\{1/\sqrt{d},1/2\}.$ There exists a $\delta=\delta(d)>0$, such that each simplex $\sigma$ of $\Xi(x,\X_n)$ is included in $C(x,R_{\delta,\eta}(x,\X_n))$. Furthermore, $\Xi(x,\X_n)$  is empty if $R_{\delta,\eta}(x,\X_n) > M_n$.
\end{lemma}
{\em Proof.}  To ease notation, denote $R_{\delta,\eta}(x,\X_n)$ by $R$. We are going to prove that all negative simplexes containing $x$ are included in $C(x,R)$, a fact that proves the two assertions of the lemma. First, if $\eta R\ge1/2$, then $C(x,R)$ contains $[0,1]^d$ and the result is trivial. So, assume that $\eta R < 1/2$, and consider  a $(q+1)$-simplex $\sigma=\{x,x_1,\dots,x_{q+1}\}$ that is not contained in $C(x,R)$. Assume without loss of generality that $x_1$ is the point in $\sigma$ maximizing the distance to $x$, which in particular means that $x_1$ is not in $C(x,R)$. The line $[x,x_1]$ hits $C(x,\eta R)$ at some cone $A_j^\delta(x, \eta R)$. 
By Lemma \ref{lem:technical_geom} and the definition of $R$,  if $A^\delta$ denotes the  union of $A_j^\delta(x,\eta R)$ and its adjacent cones in $C(x,\eta R)$, then there exists a point $z$ of $\X_n$ in $A^\delta \subset C(x,\eta R)$ and the angle $\angle xzx_1$ formed by $[z,x]$ and $[z,x_1]$ is in smaller than $c\delta$. Let us prove that all the $(q+1)$-simplexes $\sigma_t$ of the form $(\sigma \backslash \{t\}) \cup \{z\}$, for $t\in \sigma,$ have a filtration time smaller than $r(\sigma)$. If this is the case, then the cycle formed by the $\sigma_t$'s and $\sigma$ is contained in the complex at time $r(\sigma)$, meaning that $\sigma$ is not negative, concluding the proof. Therefore, it suffices to prove that $\|z-x\| \leq r(\sigma)$ and that $\|z-x_i\| \leq r(\sigma)$ for all $i$:
	\begin{itemize}
	\item  $\|z-x\| \leq \sqrt{d}\eta R \leq \|x-x_1\|$.  
	\item If $\angle xz x_1 < c\delta \le \pi/3$,
	\begin{align*}
	\|z-x_1\|^2 &= \|z-x\|^2 + \|x_1-x\|^2 -2 \langle z-x, x_1-x \rangle \\ 
	 &\leq \|z-x\|^2 + \|x_1-x\|^2 - \|z-x\|\|x_1-x\| \leq \|x_1-x\|^2 \leq r(\sigma)^2.
\end{align*}
\end{itemize}
\begin{figure}
\centering
\includegraphics[scale=.55]{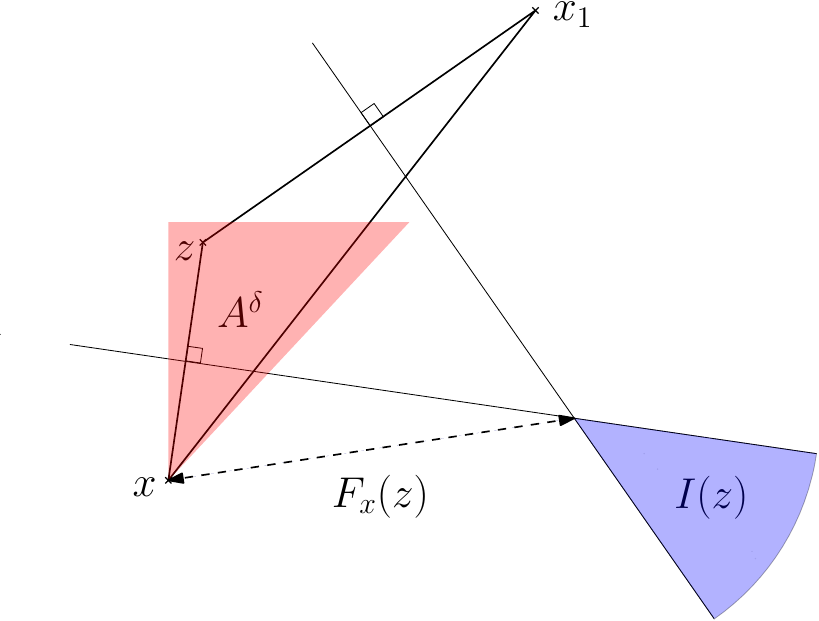}
\caption{\small \it The geometric construction used in the proof of Lemma \ref{lem:geom_lemma}. The red region represents $A^\delta$ whereas the blue region represents $I(z)$ for some point $z$ in $A^\delta$. If $\delta$ is made sufficiently small, the distance $F_x(z)$ between $x$ and $I(z)$ can be made arbitrarily large.}
\label{fig:rips_case}
\end{figure}
For $i\geq 2$, we have $\|x-x_i\| \leq \|x-x_1\|$ by assumption. Let $I(z)$ denote the set of all $t \in \R^d$ with $\|z-t\| \geq \|x-t\|$ and $\|z-t\| \geq \|x_1-t\|$, i.e.\  $I(z)$ is the intersection of two half spaces (see Figure \ref{fig:rips_case}). Let $F_x(z) = d(I(z),x)$. If we find a $\delta$ with $F_x(z) > \|x - x_1\|$ for all $z\in A^\delta$, then no $x_i$ is in $I(z)$, whatever the position of $z \in A^\delta$ is, meaning that all $x_i$'s satisfy $\|z-x_i\| \leq \max\{\|x-x_i\|, \|x_1-x_i\|\} \leq r(\sigma)$, concluding the proof. The method of Lagrange multipliers shows that $F_x(z)^2$ is a continuous function of $z$, with a known (but complicated) expression. A straightforward study of this expression shows that for $\delta$ small enough, the minimum of $F_x$ on $A^\delta$ can be made arbitrarily large: therefore, there exists $\delta$ such that $F_x(z)>\sqrt{d}\geq \|x-x_1\|,$ for all $z \in A^\delta$. 
\hfill $\square$

\subsubsection{Construction for the \texorpdfstring{\v Cech}{Cech} filtration}
A similar construction works for the \v Cech filtration, but the arguments are slightly different. First, note that each negative simplex $\sigma$ in the \v Cech filtration is such that there exists a subsimplex $\sigma^\prime$ of $\sigma$ that enters in the filtration at the same time $r(\sigma)$ as $\sigma$, and so that $r(\sigma)$ is the circumradius of $\sigma^\prime$. Then
\begin{align*}
N_q(\X_n,M) &= \sum_{\sigma \in \CC_{q+1}(\X_n)} \ones\{ \sigma \text{ negative},\ r(\sigma)\geq M\} \\
&= \sum_{\sigma \in \CC_{q+1}(\X_n)} \frac{1}{\# \sigma^\prime}\sum_{i=1}^n\ones\{ \sigma \text{ negative},\ r(\sigma)\geq M,\ X_i \in \sigma^\prime \} \\
&\leq \sum_{i=1}^n \#\Xi'(X_i,\X_n),
\end{align*}
where $\Xi'(x,\X_n)$ the set of negative $(q+1)$-simplexes $\sigma$ in the \v Cech filtration $\CC(\X_n)$ with $r(\sigma) \geq M$ and $x \in \sigma^\prime$.
\begin{lemma}\label{lem:geom_lemma_cech} For $\eta = \min\{1/\sqrt{d},1/2\}$ and some $\delta=\delta(d)>0$, each simplex $\sigma$ of $\Xi'(x,\X_n)$ is included in $C(x,R_{\delta,\eta}(x,\X_n))$. Furthermore, $\Xi'(x,\X_n)$  is empty if $R_{\delta,\eta}(x,\X_n) > M_n$.
\end{lemma}
\begin{proof}
Recall the definition of $C(x,r)$ and the partition of $C(x,r)$ into the cones $(A_j^\delta(x,r))_{j=1\dots Q}$ with corresponding  bases $(C_j^\delta(x,r))_{j=1\dots Q}$. As above, denote $R_{\delta,\eta}(x,\X_n)$ by $R$. Let $\sigma = \{x,x_1,\dots, x_{q+1}\}$ denote a $(q+1)$-simplex not included in $C(x,R)$, with $r(\sigma) \ge M$. As in the Rips case, the result is trivial if $\eta R\geq 1/2$. By definition of the \v Cech filtration, the intersection $\bigcap_{i=0}^{q+1} \overline{B}(x_i,r(\sigma))$ consists of a singleton $\{y\}$. If there is a point $z$ of $\X_n$ in $B(y,r(\sigma))$, then, by the nerve theorem applied to $\sigma \cup \{z\}$, we can conclude with similar arguments as in Lemma \ref{lem:geom_lemma} that $\sigma$ is positive in the filtration, meaning that every negative $\sigma \in \Xi(x,\X_n)$ has to be included in $C(x,R)$. 

Let us prove the existence of such a $z$. As $x \in \sigma^\prime$, the distance between $x$ and $y$ is equal to $r(\sigma) \geq R$. Therefore, the line $[x,y]$ hits $C(x,\eta R)$ in some cone $A_j^\delta(x, \eta R)$, whose base $C_j^\delta(x,r)$ intersects $[0,1]^d$, as it intersects $[x_0,y].$ As in the Rips case, there exists a point $z$ of $\X_n$ in $C(x,\eta R)$ such that the angle made by $z$, $x$ and $y$ is smaller than $c\delta$. As before, it can then be argued that $\|y-z\| \leq r(\sigma)$, concluding the proof.
  \end{proof}

\begin{remark} Note that the fact that the support of $\kappa$ is the cube only enters the picture through the geometric arguments used here and in the above proof. Some more refine work is needed to show that a similar construction holds when the cube is replaced by a convex body.
\end{remark}

\paragraph{} In the following, fix $\eta = \min\{1/\sqrt{d},1/2\},$ choose $\delta$ sufficiently small, and let $R_{\delta,\eta}(x,\X_n), A_j^\delta(x, r)$ and $C_j^\delta(x,r)$ be denoted by $R(x,\X_n), A_j(x, r)$ and $C_j(x,r)$ respectively. Both $\Xi(x,\X_n)$ and $\Xi'(x,\X_n)$ are included in the set of $(q+1)$-tuples of $\X_n \cap C(x,R(x,\X_n))$, so that the following inequality holds for either the Rips or the \v Cech filtration:
\begin{equation}
N_q(\X_n,M) \leq \sum_{i=1}^n \ones\{R(X_i,\X_n) >M_n\} \left(\#(\X_n \cap C(X_i,R(X_i,\X_n)))\right)^{q+1}.
\end{equation}
Denote $R(X_1,\X_n)$ by $R_n$. As we will see, an estimate of the tail of $R_n$ is sufficient to get a control of $N_q(\X_n,M)$. The probability $P(R_n>t)$ is bounded by the probability that one of the cones pointing at $X_1$, of radius $t/2$, wholly included in the cube $[0,1]^d,$ is empty. Conditionally on $X_1$, this probability is exactly the probability that a binomial process with parameters $n-1$ and $\kappa$ does not intersect this cone. Therefore,
\begin{equation}\label{eq:unifBoundR}
P(R_n>t) \leq c \exp(-cn t^d),
\end{equation}
and we obtain, for $\lambda>0$, 
\begin{align}
 E\left[e^{\lambda  nR_n^d}\right] &= \int_1^\infty P(\lambda  nR_n^d > \ln(t))dt \nonumber \\
 &\leq \int_1^\infty c\exp\left(-c \frac{\ln(t)}{\lambda}\right) dt < \infty \mbox{ if } \lambda < c/2. \label{eq:MGF_R}
\end{align}
\begin{lemma}\label{lem:bound_num_points}
The random variable $\#(\X_n \cap C(X_1,R_n))$ has exponential tail bounds: for $t>0$,
\begin{equation*}
P(\#(\X_n \cap C(X_1,R_n))>t )\leq c\exp(-ct).
\end{equation*}
\end{lemma}

\begin{proof}
Conditionally on $X_1$ and $R_n$, two possibilities may occur. In the first one, the cube centered at $X_1$ of radius $\eta R_n$ contains a point on its boundary, in the cone $A_{j_0}(X_1,\eta R_n)$. Denote this event $E$ and let $Q_0$ be the number of cones wholly included in the support.  The configuration of $\X_n$ is a binomial process conditioned to have at least one point in the cones $A_j(X_1,\eta R_n)$ wholly included in the cube, except for $j=j_0$, and a point on $C_{j_0}(X_1,\eta R_n)$. In this case, $\#(\X_n \cap C(X_1,R_n))$ is equal to $Q_0+ Z$, where $Z$ is a binomial variable of parameters $n-Q_0$ and 
\[\int_{C(X_1,R_n)\backslash A_{j_0}(X_1,\eta R_n)} \kappa(x)dx \leq c R_n^d.\]Therefore, for $\beta>0$, using a Chernoff bound and a classical bound on the moment generating function of a binomial variable:
\begin{align*}
P(\#(\X_n \cap C(X_1,R_n))>t|R_n,Q_0,E ) &\leq P(Q_0+Z>t|R_n,Q_0) \\
 &\leq \frac{e^{\beta Q_0}E[e^{\beta Z}|R_n]}{e^{\beta t}} \\
&\leq \frac{e^{\beta Q}e^{cnR_n^d(e^\beta-1)}}{e^{\beta t}},
\end{align*}
where $Q$ is the number of elements in the partition of $C(x,R)$.
 Take $\beta$ sufficiently small so that $E[e^{cnR_n^d(e^\beta-1)}]<\infty$ (such a $\beta$ exists by equation \eqref{eq:MGF_R}). We have the conclusion in this first case.

   The other possibility is that there exists a cone not wholly included in the cube containing no point of $\X_n$. In this case, the configuration of $\X_n$ is a binomial process conditioned on having at least one point in the cones $A_j(X_1,R_n)$ wholly included in cube and no point in a certain cone not wholly included in the cube. Likewise, a similar bound is shown.
  \end{proof}

We are now able to finish the proof of Proposition~\ref{prop:bound_size_diagrams}: for $p\geq 1$,
\begin{align*}
&E[\mu_n(U_M)^p]= E\left[\left(\frac{N_q(\X_n,M)}{n}\right)^p\right] \\
&\leq\frac{1}{(q+2)^p}E\left[\left(\frac{ \sum_{i=1}^n \ones\{R(X_i,\X_n) >M_n\} \left(\#(\X_n \cap C(X_i,R(X_i,\X_n)))\right)^{q+1}}{n}\right)^p\right] \\
&\leq \frac{1}{(q+2)^p}E\left[\ones\{R_n> M_n\} \#(\X_n \cap C(X_1,R_n))^{p(q+1)}\right] \text{ by Jensen's inequality}\\
&\leq \frac{1}{(q+2)^p}\left(P(R_n> M_n) E[\#(\X_n \cap C(X_1,R_n))^{2p(q+1)}]\right)^{1/2}.
\end{align*}
Lemma \ref{lem:bound_num_points} implies that, for $p'>0$,
\begin{align*}
E\left[\#(\X_n \cap C(X_1,R_n))^{p'}\right]&\leq \int_1^\infty ce^{-Ct^{1/p'}}dt = \frac{p'c}{C^{p'}}\int_1^\infty u^{p'-1}e^{-u}du = \frac{c}{C^{p'}} (p')!.
\end{align*}
Therefore, for $q\geq 1$,
\begin{align*}
E[\mu_n(U_M)^{p/(q+1)}] &\leq c\frac{1}{(q+2)^{p/(q+1)}}\exp(-cM^d)\left((2p)!C^{-2p}\right)^{1/2} \leq \exp(-cM^d)c^pp!.
\end{align*}
To finish the proof, we use a simple lemma relating the moments of a random variable to its tail.
\begin{lemma}\label{lem:chernof}
Let $X$ be a positive random variable such that there exists constants $A,C>0$ with
\begin{equation}
E[X^k]\leq AC^k k!.
\end{equation}
Then, there exists a constant $c>0$ such that $\forall x>0, P(X>x)\leq A \exp(-cx)$.
\end{lemma}
\begin{proof}
Fix $\lambda=\frac{1}{2C}$. The moment generating function of $X$ in $\lambda$ is bounded by:
\[ E\left[e^{\lambda X}\right] = \sum_{k\geq 0} \frac{\lambda^k E[X^k]}{k!} \leq  A\sum_{k\geq 0} \lambda^k C^k=A.\]
Therefore, using a Chernoff bound, $P(X>x) \leq A\exp(-\lambda x)$.
  \end{proof}
Apply Lemma~\ref{lem:chernof} to $X=\mu_n(U_M)^{1/(q+1)}$ to obtain the assertion of Proposition~\ref{prop:bound_size_diagrams}. 

\subsection{Proof of Theorem \ref{thm:main_cv}}
\subsubsection{Step 1: Convergence for functions vanishing on the diagonal}

The first step of the proof is to show that the convergence holds  $C_0(\upperdiag)$, the set of continuous bounded functions vanishing of the diagonal.  The crucial part of the proof consists in using Proposition \ref{prop:bound_size_diagrams}, which bounds the total number of points in the diagrams. An elementary lemma from measure theory is then used to show that it implies the a.s. convergence for vanishing functions. We say that a sequence of measures $(\mu_n)_{n\geq 0}$ converges $C_0$-vaguely to $\mu$ if $\mu_n(\phi)\to \mu(\phi)$ for all functions $\phi$ in $C_0(\upperdiag)$.

\begin{lemma}\label{lem:measureLemma}
Let $E$ be a locally compact Hausdorff space. Let $(\mu_n)_{n\geq 0}$ be a sequence of Radon measure on $E$ which converges $C_c$-vaguely to some measure $\mu$. If $ \sup_{n}|\mu_n| < \infty$, then $(\mu_n)_{n\geq 0}$ converges $C_0$-vaguely to $\mu$.
\end{lemma}

\begin{proof}Let $(h_q)$ be a sequence of functions with compact support converging to $1$ and let $\phi \in C_0(E)$. Fix $\varepsilon >0$. By definition of $C_0(E)$, there exists a compact set $K_\varepsilon$ such that $f$ is smaller than $\varepsilon$ outside of $K_\varepsilon$. For $q$ large enough, the support of $h_q$ includes $K_\varepsilon$. Let $\phi_q = \phi \cdot h_q$. Then,
\begin{align*}
|\mu_n(\phi)-\mu(\phi)| &\leq |\mu_n(\phi) - \mu_n(\phi_q)| + |\mu_n(\phi_q) - \mu(\phi_q)| + |\mu(\phi_q) - \mu(\phi)| \\
&\leq (\sup_{n}|\mu_n| +|\mu|) \varepsilon + |\mu_n(\phi_q) - \mu(\phi_q)|.
\end{align*}
As $(\mu_n)_n$ converges vaguely to $\mu$, the last term of the sum converges to $0$ when $\varepsilon$ is fixed. Hence, we have $\limsup_{n\to \infty}|\mu_n(\phi)-\mu(\phi)| \leq \left(\sup_{n}|\mu_n| +|\mu|\right) \varepsilon $. As this holds for all $\varepsilon > 0$, $\mu_n(\phi)$ converges to $\mu(\phi)$.
  \end{proof}

Taking $M=0$ in Proposition \ref{prop:bound_size_diagrams}, we see that $\sup_n E[|\mu_n|] <\infty$. Therefore, the $C_0$-vague convergence of $E[\mu_n]$ is shown in the binomial setting. To show that the convergence also holds almost surely for $|\mu_n|$, we need to show that $\sup_n |\mu_n| <\infty$. For this, we use concentration inequalities. We do not show concentration inequalities for $|\mu_n|$ directly. Instead, we derive concentration inequalities for $\sum_{i=1}^n \#(\X_n \cap C(X_i,R(X_i,\X_n)))^{q+1}$, which is a majorant of $|\mu_n|$. Recall that $R(X_i,\X_n)$ is defined as the smallest radius $R$ such that, for some fixed parameter $\eta>0$, and for each $j=1\dots Q$, $A_j(X_i,\eta R)$, either contains a point of $\X_n$ different than $X_i$, or is not contained in the cube. To ease the notations, we denote $R(X_i,\X_n)$ by $R_{i,n}$.

\begin{lemma}\label{lem:concentration_bound}
Fix $M \geq 0$ and define $Z_n^M = \sum_{i=1}^n \#(\X_n \cap C(X_i,R_{i,n}))^q \ones\{R_{i,n}\geq M\}$. Then, for every $\varepsilon >0$, there exists a constant $c_\varepsilon >0$ such that
\begin{equation}\label{eq:concentration_bound}
    P(|Z_n^M-E[Z_n^M]|>t) \leq \frac{n^{\frac{3}{2}+\varepsilon}}{t^3}\exp(-c_\varepsilon n^{-1}M^d).
\end{equation}
The constant $c_{\varepsilon}$ depends on $\varepsilon,d,q$ and $\kappa$.
\end{lemma}
 As a consequence of the concentration inequality, $n^{-1}Z_n^0$ is almost surely bounded. Indeed, choose $\varepsilon < 1/2$:
\begin{align*}
P(n^{-1}|Z_n^0-E[Z_n^0]|>t) \leq \frac{n^{\frac{3}{2}+\varepsilon}}{(nt)^3} = \frac{n^{-\frac{3}{2}+\varepsilon}}{t^3}.
\end{align*}
By Borel-Cantelli lemma, almost surely, for $n$ large enough, we have  $n^{-1}|Z_n^0-E[Z_n^0]|\leq 1$. Moreover, $\sup_n n^{-1}E[Z_n^0]$ is finite. As a consequence, $\sup_n n^{-1}Z_n^0$ is almost surely finite. As this is an upper bound of $\sup_n |\mu_n|$, we have proven that $\sup_n |\mu_n| < \infty$ almost surely. By Lemma \ref{lem:measureLemma}, the sequence $\mu_n$ converges $C_0$-vaguely to $\mu$. The proof of Lemma \ref{lem:concentration_bound} is based on an inequality of the Efron-Stein type and is rather long and technical. It can be found in Section \ref{sec:technical_lemma}.

We now briefly consider the Poisson setting. Define $\mu_n' = \mu_{N_n} \times (N_n/n)$, where $(N_i)_{i\geq 1}$ is some sequence of independent Poisson variables of parameter $n$, independent of $(X_i)_{i \geq 1}$. \\[5pt]
%
$\bullet$\;\; $E[|\mu'_n|] = E\left[\frac{N_n}{n}E[|\mu_{N_n}||N_n]\right] \leq \sup_n E[|\mu_n|] <\infty$. Therefore, $C_0$-convergence of the expected diagram holds in the Poisson setting.\\[3pt]
$\bullet$\;\; Likewise, it is sufficient to show that $\sup_n \frac{N_n}{n} <\infty$ to conclude to the $C_0$-convergence of the diagram in the Poisson setting. Fix $t>1$. It is shown in the chapter 1 of the monograph \cite{penrose2003random} that $P(N_n>nt) \leq \exp(-nH(t))$, where $H(t) = 1-t+t\ln t$. This gives us
\begin{align}
P\left(\sup_n \frac{N_n}{n} \leq t\right) &= \prod_{n\geq 0} (1-P(N_n>nt)) \nonumber \\
&\geq \prod_{n\geq 0}(1-\exp(-nH(t))) = \exp\left(\sum_{n\geq 0} \ln(1-\exp(-nH(t)))\right)  \label{tedious}
\end{align}
The series $\sum_n \ln(1-x^n)$ is equal to $-\sum_n \sigma(n)/n x^n$ when $|x|<1$, and where $\sigma(n)$ is the sum of the proper divisors of $n$. Therefore it is a power series, and is continuous on $]-1,1[$. Since $t$ tends to infinity, $\exp(-H(t))$ converges to $0$, and thus the quantity appearing in the right hand side of \eqref{tedious} converges to $1$ as $t$ tends to infinity.

\subsubsection{Step 2: Convergence for functions with polynomial growth}
The second step consists in extending the convergence to functions $\phi \in C_{\mathrm{poly}}(\upperdiag)$. We only show the result for binomial processes. The proof can be adapted to the Poisson case using similar techniques as at the end of Step 1. The core of the proof is a bound on the number of points in a diagram with high persistence. For $M>0$, define $T_M = \{\rD=(r_1,r_2)\in \upperdiag \mbox{ s.t. } \pers(\rD) \geq M\}$. Let $P_n(M) = n\mu_n(T_M)$ denote the number of points in the diagram with persistence larger than $M$.

 First, we show that the expectation of $P_n(M)$ converges to $0$ at an exponential rate when $M$ tends to $\infty$. The random variable $P_n(M)$ is bounded by $n\mu_n(U_M)$. By Proposition \ref{prop:bound_size_diagrams}, recalling that $q$ is the degree of homology,
\begin{align}
E\big[P_n(M)\big] &\leq \int_0^\infty P\big(n\mu_n(U_M) \geq t \big)dt  \nonumber\\ 
&\leq \int_0^\infty c\exp\big(-c(M^d+(t/n)^{1/(q+1)}\big)dt \nonumber\\
&\leq cn\exp(-cM^d) \int_0^\infty \exp(-u) qu^qdu =cn\exp\big(-cM^d\big) \label{eq:bound_PnM}
\end{align}

Fix a sequence $(g_M)$ of continuous functions with support inside the complement of $T_M$ taking their values in $[0,1]$, equal to $1$ on $T_{M-1}^c$. Let $\phi$ be a function with polynomial growth, i.e.\ satisfying \eqref{polyGrowth} for some $A,\alpha>0$. Define $\phi_M = \phi \cdot g_M$. We have the decomposition:
\begin{equation}
    E[\mu_n(\phi)] = (E[\mu_n(\phi)]-E[\mu_n(\phi_M)]) + E[\mu_n(\phi_M)]. 
\end{equation}
As $\phi_M \in C_0(\upperdiag)$, the second term on the right converges to $\mu(\phi_M)$. The first term on the right is bounded by 
%
\begin{align}
E[\mu_n(\phi)]-E[\mu_n(\phi_M)] &\leq E[\mu_n(A(1+\pers^\alpha)(1-g_M))] \nonumber \\
&\leq AE[P_n(M)]/n + AE[\Pers_\alpha(M;\dgmmap[\X_n])]/n \nonumber \\
&\leq cA \exp\left(-cM^d \right) + AE[\Pers_\alpha(M;\dgmmap[\X_n])]/n, \label{ineqTHM3}
\end{align}
using inequality \eqref{eq:bound_PnM}. It is shown in \cite{cohen2010lipschitz} that
\[ \Pers_\alpha(M;\dgmmap[\X_n]) \leq M^\alpha P_n(M) + \alpha \int_{M}^\infty P_n(\varepsilon)\varepsilon^{\alpha-1}d\varepsilon.\]
Hence, by Fubini's theorem and inequality \eqref{eq:bound_PnM}:
\begin{equation}\label{ineqTHM3.2}
E[\Pers_\alpha(M;\dgmmap[\X_n])]/n \leq c M^\alpha \exp\left(-cM^d \right) + c \alpha \int_{M}^\infty \exp\left(-c \varepsilon^d \right)\varepsilon^{\alpha-1}d\varepsilon.
\end{equation} 
and this quantity goes to $0$ as $M$ goes to infinity.  Moreover, applying this inequality to $M=0$, we get that $C_0=\sup_n E[\mu_n(\phi)] < \infty$. Therefore, $\lim_{n\to \infty} E[\mu_n(\phi_M)] = \mu(\phi_M) \leq C_0$. By the monotone convergence theorem, $\mu(\phi_M)$ converges to $\mu(\phi)$ when $\phi$ is non negative, with $\mu(\phi)$ finite by the latter inequality. If $\phi$ is not always non negative, we conclude by separating its positive and negative parts. Finally, looking at the bounds \eqref{ineqTHM3} and \eqref{ineqTHM3.2},
\begin{align*}
\limsup_{n\to \infty} |E[\mu_n(\phi)]-\mu(\phi)| &\leq \limsup_{n\to\infty} (E[\mu_n(\phi)]-E[\mu_n(\phi_M)]) + |\mu(\phi_M)-\mu(\phi)| \to_{M\to \infty} 0.
\end{align*}

We now prove that $\mu_n(\phi) - E[\mu_n(\phi)]$ converges a.s. to $0$. Similar to the above, it is enough to show that $P_n(M)$ is almost surely bounded by a quantity independent of $n$, which converges to $0$ at an exponential rate when $M$ goes to $\infty$. The random variable $(q+2)P_n(M)$ is bounded by $Z_n^M$, which is defined in Lemma \ref{lem:concentration_bound}, and whose expectation is controlled. Therefore, it remains to show that $Z_n^M$ is close to its expectation. We have
 \begin{align}
 &\limsup_{n\to \infty} |\mu_n(\phi) - E[\mu_n(\phi)]| \leq \limsup_{n\to \infty}\left(|\mu_n(\phi-\phi_M)|+ |(\mu_n-E[\mu_n])(\phi_M)| + |E[\mu_n](\phi-\phi_M)| \right) \nonumber \\
 &\leq \limsup_{n\to \infty} |\mu_n(\phi-\phi_M)| +0+cA \exp\left(-cM^d \right)+ c M^\alpha \exp\left(-cM^d \right) + c \alpha \int_{M}^\infty \exp\left(-c \varepsilon^d \right)\varepsilon^{\alpha-1}d\varepsilon \label{eq:boundas}
 \end{align}
 by inequalities \eqref{ineqTHM3} and \eqref{ineqTHM3.2}. The random variable $|\mu_n(\phi-\phi_M)|$ is bounded by 
\begin{align}
An^{-1}&\left(P_n(M) + M^\alpha P_n(M)+ \alpha \int_{M}^\infty P_n(\varepsilon)\varepsilon^{\alpha -1}d\varepsilon.\right) \nonumber \\
&\leq An^{-1} \left(Z_n^{M_n} + M^\alpha Z_n^{M_n}+ \alpha \int_{M}^\infty Z_n^{\varepsilon_n} \varepsilon^{\alpha-1}d\varepsilon,\right) \label{ineqTHM3.3}
\end{align} 
where $M_n=n^{-1/d}M$ and $\varepsilon_n = n^{-1/d}\varepsilon$. As a consequence of Lemma \ref{lem:concentration_bound}, by choosing $\varepsilon$ so that $-3/2+\varepsilon <-1$, 
\[P\left(\sup_n n^{-1}|Z_n^{M_n}-E[Z_n^{M_n}]|>t \right) \leq c\frac{\exp(-cM^d)}{t^3}.\]
Fixing $t=\exp(-(c/6)M^d)$ and using Borel-Cantelli lemma, for $M \in \mathbb{N}$ large enough, $\sup_n n^{-1}|Z_n^{M_n}-E[Z_n^{M_n}]| < \exp(-(c/6)M^d)$.
Also, $E[Z_n^{M_n}] \leq  n c \exp(-cM^d)$. Therefore, for $\alpha\geq 0$, \[\lim_{M\to \infty} \limsup_{n \to \infty} n^{-1}M^\alpha Z_n^{M_n} = 0 \mbox{ a.s.}\] The third term in the sum \eqref{ineqTHM3.3} is less straightforward to treat. As $Z_n^M$ is a decreasing function of $M$, for $M\in \mathbb{N}$ large enough and with $k_n = n^{-1/d}k$:
\begin{align*}
\int_{M}^\infty n^{-1}Z_n^{\varepsilon_n}\varepsilon^{\alpha-1}d\varepsilon &\leq \sum_{k \geq M}n^{-1}Z_n^{k_n} \int_{k}^{k+1} \varepsilon^{\alpha-1}d\varepsilon \\
&\leq \sum_{k \geq M}c\exp(-ck^d) \frac{1}{\alpha} \left((k+1)^\alpha-k^\alpha\right) \\
&\leq \sum_{k \geq M}c\exp(-c k^d) \frac{k^\alpha 2^\alpha}{\alpha} = o(1).
\end{align*}
As a consequence, $\lim_{M\to \infty} \limsup_n |\mu_n(\phi-\phi_M)| = 0$. As the three last terms appearing in inequality \eqref{eq:boundas} also converges to $0$ when $M$ goes to infinity, we have proven that $\mu_n(\phi)-E[\mu_n(\phi)]$ converges a.s. to $0$. Therefore, $\mu_n(\phi)$ converges a.s. to $\mu(\phi)$.

Finally, we have to prove assertion (ii) in Theorem \ref{thm:main_cv}, i.e.\ that the convergence holds in $L_p$. As the convergence holds in probability, it is sufficient to show that $(\mu_n(\phi)^p)_n$ is uniformly integrable. Observing that $E[\mu_n(\phi)^p \ones\{\mu_n(\phi)^p >M\}]  \leq \left( E[\mu_n(\phi)^{2p}] P(\mu_n(\phi)^p>M)\right)^{1/2}$, uniform integrability follows from $\sup_n E[\mu_n(\phi)^{p}] <\infty$ for any $p>1$. We have
\begin{align*}
 \mu_n(\phi)^p &\leq \mu_n(A(1+\pers^\alpha))^p \leq 2^{p-1}(A^p|\mu_n|^p + \mu_n(\pers^\alpha)^p),
\end{align*}
and from Proposition \ref{prop:bound_size_diagrams} we easily obtain that $E[|\mu_n|^p]$ is uniformly bounded. We treat the other part by assuming without loss of generality that $p$ is an integer:
\begin{align*}
E[\mu_n(\pers^\alpha)^p] &\leq \frac{1}{n^p}E\left[ \left(\int_0^\infty P_n(\varepsilon) \varepsilon^{\alpha-1} d\varepsilon\right)^p \right]\\
 &= \frac{1}{n^p}E\left[ \int_{[0,\infty[^p}P_n(\varepsilon_1) \cdots P_n(\varepsilon_p) (\varepsilon_1\cdots \varepsilon_p)^{\alpha-1} d\varepsilon_1\cdots \varepsilon_p\right]\\
 &= \frac{1}{n^p} \int_{[0,\infty[^p}E[P_n(\varepsilon_1) \cdots P_n(\varepsilon_p) ](\varepsilon_1\cdots \varepsilon_p)^{\alpha-1} d\varepsilon_1\cdots \varepsilon_p \\
 &\leq \frac{1}{n^p} \int_{[0,\infty[^p}\left(E[P_n(\varepsilon_1)^p] \cdots E[P_n(\varepsilon_p) ^p]\right)^{1/p}(\varepsilon_1\cdots \varepsilon_p)^{\alpha-1} d\varepsilon_1\cdots \varepsilon_p\\
 &\leq \frac{1}{n^p} \int_{[0,\infty[^p}\left(n^{p^2} c\exp(-c(\varepsilon_1^d+\cdots + \varepsilon_p^d))\right)^{1/p}(\varepsilon_1\cdots \varepsilon_p)^{\alpha-1} d\varepsilon_1\cdots \varepsilon_p\\
  &= c \int_{[0,\infty[^p}\exp\left(-\frac{c}{p}(\varepsilon_1^d+\cdots + \varepsilon_p^d)\right)(\varepsilon_1\cdots \varepsilon_p)^{\alpha-1} d\varepsilon_1\cdots \varepsilon_p\\
  &= c \left( \int_0^\infty \exp\left(-\frac{c}{p}\varepsilon^d\right) \varepsilon^{\alpha-1}d\varepsilon \right)^p <\infty.
\end{align*}

\subsection{Proof of Proposition \ref{prop:density}}

The proof relies on the regularity of the number of simplexes appearing at certain scales.
\begin{lemma}\label{lem:LLNN_k} Let $q\geq 0$. For $r_1<r_2$, let $F_q(\X_n,r_1,r_2)$ be the number of $q$-simplexes $\sigma$ in the filtration $\KK(\X_n)$ with $r(\sigma)\in [r_1,r_2]$. Assume that $\X_n$ is a binomial $n$-sample of density $\kappa$. Then,
\begin{equation}\label{appliLLN}
 n^{-1}F_q(n^{1/d}\X_n,r_1,r_2) \xrightarrow[n\to \infty]{L_2} F_q(r_1,r_2),
\end{equation}
where $F_q(r_1,r_2)\leq cr_2^{2dq-1} |r_2-r_1|$.
\end{lemma}

\begin{proof} For a finite set $\X \subset \R^d$, define
\[\xi^{r_1,r_2}(x,\X)= \frac{1}{q+1} \sum_{\sigma \in \KK_q(\X)} \ones\{r(\sigma)\in [r_1,r_2] \mbox{ and } x\in \sigma\}.\]
Then, $F_q(\X_n,r_1,r_2) = \sum_{x\in \X} \xi^{r_1,r_2}(x,\X_n)$. The paper \cite{penroseLLN} shows convergence in $L_2$ of such functionals $\xi(x,\X)$ under two conditions. The first one of them is called stabilization. Let $\P$ be a homogeneous Poisson process in $\R^d$. A quantity $\xi(x,\X)$ is stabilizing if, with probability one, there exists some random radius $R<\infty$ such that, for all finite sets $A$ which are equal to $\P$ on $B(0,R)$,
\[ \xi(0,\P\cap B(0,R)) = \xi(0,A),\]
Denote this quantity by $\xi_\infty(\P)$. In our case, $\xi^{r_1,r_2}$ is stabilizing with $R=2r_2$. The second condition is a moment condition: there exists some number $\beta>2$ such that
\[ \sup_n E[\xi(n^{1/d}X_1,n^{1/d}\X_n)^\beta]<\infty .\]
Once again, $\xi^{r_1,r_2}$ possesses this property: the random variable $\xi^{r_1,r_2}(n^{1/d}X_1,n^{1/d}\X_n)$ is bounded by the number of $q$-simplexes of $\KK(\X_n)$ containing $X_1$ and being included in $B(X_1,2n^{-1/d}r_2)$. This number of $q$-simplexes is bounded by $\#(\X_n \cap B(X_1,2n^{-1/d}r_2))^{q}$, which, in turn, is stochastically dominated by a binomial random variable with parameters $n$ and $cn^{-1}r_2^d$. In particular, its moment of order $3q$ is smaller than a constant independent of $n.$ This means that the moment condition is satisfied. Applying the main theorem of \cite{penroseLLN}, convergence \eqref{appliLLN} is obtained, with $F_q(r_1,r_2)=E[\xi_\infty^{r_1,r_2}(\P)]$, where $\xi^{r_1,r_2}_{\infty}(\P) = \xi^{r_1,r_2}(0,\P\cap B(0,2r_2))$. The set $\P\cap B(0,2r_2)$ can be expressed as $\{X_1,\dots,X_N\},$ where $(X_i)_{i\geq 0}$ is a sequence of i.i.d.\ uniform random variables on $B(0,2r_2)$, and $N$ is an independent Poisson variable with parameter $c r_2^d$. Therefore,
\begin{align*}
E[\xi^{r_1,r_2}_{\infty}(\P)] &= E\left[ \sum_{i_1,\dots,i_q}  \ones\{ r(\{0,X_{i_1},\dots,X_{i_q}\}) \in [r_1,r_2]\} \right] \\
&= E\left[\frac{N!}{(N-q)!}\right] P(r(\{0,X_1,\dots,X_q\}) \in [r_1,r_2])\\
&\leq c r_2^{2dq-1}|r_2-r_1|.
\end{align*} 
The last inequality is a consequence of (i) the fact that the $q$-th factorial moment of $N$ equals $cr_2^{dq},$ and (ii) of the following lemma.
  \end{proof}

\begin{lemma}\label{lem:density_cech_bounded}
If $X_1,\dots,X_q$ is a $q$-sample of the uniform distribution on $B(0,2)\subset \R^d$, and $r$ is either the filtration time of the \v Cech or Rips filtration, then, for any $0<a<b\leq 2$, 
\begin{equation}
P(r(0,X_1,\dots,X_q)\in [a,b]) \leq C_{q,d}|a-b|,
\end{equation}
for some constant depending on $d$ and $q$.
\end{lemma}

\begin{proof}
Having such an inequality is equivalent to having the filtration time $r(0,X_1,\dots,X_q)$ having a bounded density on $[0,2]$. We treat separately the case of the Rips and of the \v Cech filtration.

\paragraph*{\textbf{Rips filtration}} 
\ The quantity $r(0,X_1,\dots,X_q)$ is equal to $|X_i|$ or $|X_i-X_j|$ for some indexes $i,j$. Hence,
one has $P(r(0,X_1,\dots,X_q)\in [a,b])\leq q P(|X_1|\in [a,b]) + \frac{q(q-1)}{2} P(|X_1-X_2|\in [a,b])$. The random variables $|X_1|$ and $|X_1-X_2|$  have bounded densities on $[0,2]$, so that the result follows.

\paragraph*{\textbf{\v Cech filtration}}\  Let $r'(x_1,\dots,x_k)$ be the radius of the circumsphere of $x_1,\dots,x_k$. Then, $r(x_1,\dots,x_q)=r'(x_{i_1},\dots,x_{i_k})$ for a certain subset of $\{1,\dots,q\}$. Hence,
\begin{equation}
P(r(0,X_1,\dots,X_q)\in [a,b]) \leq \hspace{-.2cm}\sum_{\sigma\subset\{1,\dots,q\}}( P(r'(0,X_\sigma)\in [a,b]) +P(r'(X_\sigma)\in [a,b]) ).
\end{equation}
We are going to show that $r'(X_\sigma)$ has a bounded density on $[0,2]$ by induction on $k$, and it is then shown likewise that $r'(0,X_\sigma)$ has a bounded density. For $k=2$, $r'(X_\sigma)$ is the distance between $X_1$ and $X_2$, which  has a bounded density. If $k>2$, we let $r_k$ be the circumradius of $\{X_1,\dots,X_k \}$, $r_{k-1}$ the circumradius of $\{X_2,\dots,X_k\}$, with associated circumcenters $z_k$, $z_{k-1},$ respectively, and $U$ 
be  the affine $(k-2)$-dimensional space spanned by  $\{X_2,\dots,X_k\}$.  The vector  $z_k-z_{k-1}$ is orthogonal to $U$ and therefore $r_k^2=|z_k-z_{k-1}|^2+r_{k-1}^2$. 
For any subspace $E$, we let $\pi_E$ be the orthogonal projection onto $E$, $E^\bot$ be the orthogonal space from $E$ and $S_E$ be the unit sphere in $E$. Without loss of generality we assume that $z_{k-1}=0$, so that $U$ is a subspace of $\R^d$. For any $\theta \in S_{U^\bot}$, we let $V(\theta) = U+\R \theta$. 
Let $\theta_0$ be any vector in $S_{U^\bot}$, with $V_0 = V(\theta_0)$ and introduce the function $\Phi : \R^+ \times S_{U^\bot}\times S_{V_0} \to \R^d$ defined by
\[ \Phi(t,\theta,x) = t\theta + \sqrt{t^2+r_{k-1}^2}R(\theta)x, \]
 where $R(\theta)$ is an isometry from $V_0$ to $V(\theta)$ defined by $R(\theta)x = \pi_{U}(x) + (x \cdot \theta_0)\theta$ for $x\in V_0$. Notice that, for each $\theta \in S_{U^\top}$, we have $\Phi(t,\theta,x) \in V(\theta).$ See also Figure \ref{fig:construction}.
\begin{figure}
\centering
\includegraphics[scale=.6]{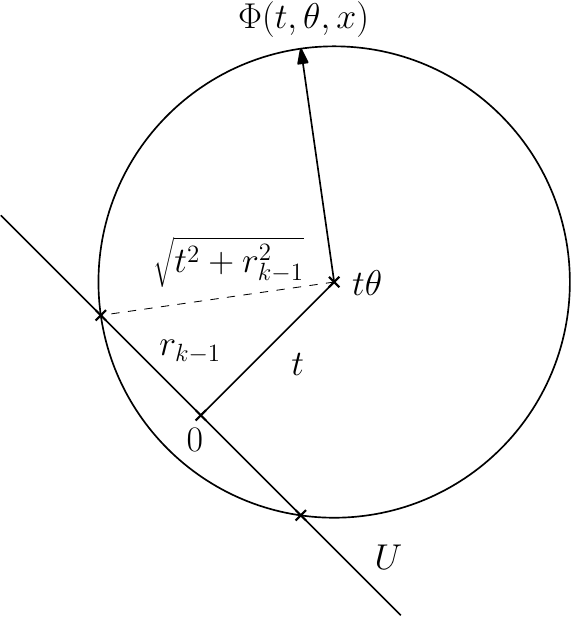}
\caption{Definition of $\Phi$. We display the plane $V(\theta)$.}\label{fig:construction}
\end{figure}

{\bf Fact.} {\em
The function $\Phi$ is injective and we have $\{r_k\in [a,b]\} \subset \{ X_1\in \Phi(A_{r_{k-1}})\}$, where $A_{r_{k-1}}= M_{a,b}\times S_{U^\bot}\times S_{V_0}$ and $M_{a,b} = \{t\geq 0,\ a^2\leq t^2+r_{k-1}^2\leq b^2\}$.}\\

The proof of this fact is given below. We continue the proof of the lemma by assuming that the fact holds.  
%
Letting $\lambda_d$ denote the Lebesgue measure on $\R^d$, and $c_d^{-1}$ the $d$-dimensional volume of $B(0,2),$ we have
\begin{align}
&P(r'(X_\sigma)\in [a,b]) = E[P(r_k \in [a,b]|r_{k-1})] \le E\left[ P(X_1\in \Phi(A_{r_{k-1}})|r_{k-1}) \right]\nonumber \\
&\leq c_dE\left[ \lambda_d(\Phi(A_{r_{k-1}})  )  \right] = c_dE\left[ \int_{A_{r_{k-1}}} J\Phi(t,\theta_1,\theta_2)\dd t \dd \theta_1\dd\theta_2 \right]  \label{eq:need_to_bound_jacob}
\end{align}

Let us compute the Jacobian $J\Phi(y)$ of $\Phi$ at some point $y=(t,\theta,x)\in A_{r_{k-1}}$. The tangent space of $S_{U^\bot}$ at $\theta$ is equal to $U^\bot \cap (\R\theta)^\bot=V(\theta)^\bot$ and the tangent space of $S_{V_0}$ at $x$ is equal to $V_0 \cap (\R x)^\bot$. We compute the partial derivatives:
\begin{align*}
&\partial_t \Phi(y)[h_0] = \theta h_0 + \frac{th_0}{ \sqrt{t^2+r_{k-1}^2}}R(\theta)x, &&h_0\in \R\\
&\partial_{\theta} \Phi(y)[h_1]= h_1t+ \sqrt{t^2+r_{k-1}^2}(x \cdot \theta_0)h_1 &&h_1\in V(\theta)^\bot \\
&\partial_{x}\Phi(y)[h_2] = \sqrt{t^2+r_{k-1}^2}R(\theta)h_2 && h_2\in  V_0 \cap (\R x)^\bot.
\end{align*}
We decompose the space $\R^d$ as follows. Let $g =R(\theta)x \in V(\theta)$, $G=(\R g)^\bot \cap V(\theta)$, and $H=V(\theta)^\bot$. Then, $\pi_{\R g} + \pi_G + \pi_H= \id$ (recall that $\pi_E$ denotes the orthogonal projection onto $E$). Also, note that, as $h_2\in  V_0 \cap (\R x)^\bot$ is orthogonal to $x,$ and $R(\theta)$ is an isometry, the vector $R(\theta)h_2$ is orthogonal to $g=R(\theta)x$. We have with $s_{k-1} = \sqrt{t^2 + r_{k-1}^2}$ that
\begin{alignat*}{3}
    &\pi_{\R g}\partial_t\Phi(y)[h_0] =\left(( \theta \cdot g)h_0 +  \frac{th_0}{ s_{k-1}} \right)g
; \  &&\pi_{\R g}\partial_{\theta} \Phi(y)[h_1]= 0 
; \  &&\pi_{\R g}\partial_{x}\Phi(y)[h_2] = 0 \\
  &\pi_{G}\partial_t\Phi(y)[h_0] = \pi_G\theta h_0; \  && \pi_{G}\partial_{\theta} \Phi(y)[h_1]= 0 
 ; \  && \pi_{G}\partial_{x}\Phi(y)[h_2] = s_{k-1}\pi_G R(\theta)h_2 \\
 &\pi_{H}\partial_t\Phi(y)[h_0] = 0;  \  && \pi_{H}\partial_{\theta} \Phi(y)[h_1]=  h_1t+ s_{k-1}(x \cdot \theta_0)h_1;  \ &&\pi_{H}\partial_{x}\Phi(y)[h_2] = 0 .
\end{alignat*}
Hence, remarking that $\pi_GR(\theta)$ is an isometry from $V_0\cap (\R x)^\bot$ to $G=V(\theta)\cap(\R g)^\bot$,
\begin{align*}
J \Phi(y) &=  \left|\theta\cdot g  + \frac{t}{ s_{k-1}} \right| \times \left|\det\p{ s_{k-1}\pi_G R(\theta)_{|V_0\cap(\R x)^\bot}}\right|\times \left|\det\p{\p{t + s_{k-1} x \cdot \theta_0 }id_{H}}\right|\\
&\leq 2 \times  s_{k-1}^{k-2} \times \p{t + s_{k-1}x \cdot \theta_0}^{d-k+1} \\
&\leq 2^{d-k+2}s_{k-1}^{d-1}.
\end{align*}
 Therefore, letting $t_0= \sqrt{a^2-r_{k-1}^2}$ and $t_1= \sqrt{b^2-r_{k-1}^2}$, we may bound \eqref{eq:need_to_bound_jacob} as follows
\begin{align*}
P(r'(X_\sigma)\in [a,b]) &\leq c_dE\left[ \int_{A_{r_{k-1}}} 2^{d-k+2}\p{t^2+r_{k-1}^2}^{(d-1)/2}\dd t \dd \theta_1\dd\theta_2 \right] \\
& \leq c_{d,k} E\left[\int_{t_0}^{t_1}(t^2 + r_{k-1}^2)^{(d-1)/2}\dd t \right]
\end{align*}
Introduce $F:x \mapsto \int_0^{\sqrt{x^2-r_{k-1}^2}}(t^2 + r_{k-1}^2)^{(d-1)/2}\dd t$. Then, \[F'(x) =x^{d-1} \times \frac{x}{\sqrt{x^2-r_{k-1}^2}}= \frac{x^{d}}{\sqrt{x^2-r_{k-1}^2}} \]
and, letting $f_{k-1}$ be the density of $r_{k-1}$ on $[0,2]$,
\begin{align*}
P(r'(X_\sigma)\in [a,b])&\leq  c_{d,k} E\left[\int_{a}^{b}F'(x)\dd x \right] \\
&\leq c_{d,k}  \int_{a}^{b}\int_0^b \ones\{x\geq r_{k-1}\}f_{k-1}(r_{k-1})\frac{x^{d}}{\sqrt{x^2-r_{k-1}^2}}\dd x \dd r_{k-1}\\
&\leq c'_{d,k}  \int_{a}^{b}x^{d}\int_0^x \frac{1}{\sqrt{x^2-r_{k-1}^2}} \dd r_{k-1}\dd x \\
&\leq c'_{d,k,}  \int_{a}^{b}x^{d}\int_0^1 \frac{1}{\sqrt{1-u^2}}\dd u \dd x \leq c''_{d,k} |b-a|,
\end{align*}
where at the last line we used that the function $x\mapsto x^{d}$ is bounded on $[0,2]$. Hence, $r'(X_\sigma)$ has a bounded density on $[0,2]$, and the induction step is proven. It remains to verify the Fact. 

{\em Proof of Fact.}
We first prove the injectivity. Let $v=\Phi(t,\theta,x)$ for some $y=(t,\theta,x)\in  \R^+ \times S_{U^\bot}\times S_{V_0}$. Then, $\pi_{U^\bot}(y)$ is colinear with $\theta$, so that $\theta$ is determined up to a sign by $y$. Let $S_1$ be the sphere in $U$, centered at $0$, of radius $r_{k-1}$, and let $S_2$ be the unique sphere in $V(\theta)$ containing both $y$ and $S_1$. Then $y\in S_2$, while the center of $S_2$ is $t\theta$, so that $t$ and $\theta$ are uniquely determined by $y$.
It follows that $R(\theta)x=\frac{y-t\theta}{\sqrt{t^2+r_{k-1}^2}}$ is also uniquely determined by $y$, and so is $x$, showing the injectivity of $\Phi$.

Let $t=|z_k-z_{k-1}|$ and $\theta= (z_k-z_{k-1})/t$. As $z_k-z_{k-1}$ is orthogonal to $U$, we have $\theta \in S_{U^\bot}$. The point $X_1$ lies inside the sphere of the space spanned by $U$ and $z_k-z_{k-1}$, centered at $z_k$, of radius $r_k = \sqrt{t^2+r_{k-1}^2} \in [a,b]$. Therefore, $X_1= t \theta + \sqrt{t^2+r_{k-1}^2}y$, where $y$ is some unit vector in $V(\theta)$, which can be written as $R(\theta)x$ for some $x\in V_0$. Hence, $X_1\in \Phi(A_{r_{k-1}})$. This completes the proof of the Fact and of the lemma.
 \end{proof}

We may now prove Proposition \ref{prop:density}. Fix $0<r_1<r_2$. We wish to show that, as $r_1$ and $r_2$ get closer, $\pi^\star_1 \mu(]r_1,r_2[)$ goes to $0$. By the Portemanteau Theorem, $\pi^\star_1 \mu([r_1,r_2]) \leq \liminf_n \pi^\star_1 \mu_n(]r_1,r_2[)$. It is shown in Lemma \ref{lem:LLNN_k} that this quantity is smaller than $cr_2^{2dq-1}|r_2-r_1|$, a quantity which converges to $0$ when $r_2$ goes to $r_1$. A similar proof holds for $\pi^\star_2 \mu$.


\subsection{Proof of Lemma \ref{lem:concentration_bound}}\label{sec:technical_lemma}
The lemma is based on an inequality of the Efron-Stein type, combined with Markov's inequality.

\begin{theorem}[Theorem 2 in \cite{boucheron2005moment}]\label{massart} Let $\X$ be a measurable set and $F:\X^n \to \R$ a measurable function. Define a $n$-sample $\X_n = \{X_1,\dots,X_n\}$ and let $Z = F(\X_n)$. If $\X_n'$ is an independent copy of $\X_n$, denote $Z_i' = F(X_1,\dots,X_{i-1},X_i',X_{i+1},\dots,X_n)$. Define 
\[ V = \sum_{i=1}^n E[(Z-Z_i')^2|\X_n]. \]
Then, for $p \geq 2$, there exists a constant $C_p$ depending only on $p$ such that
\[E[|Z-E[Z]|^p] \leq C_p E[V^{p/2}]. \]
\end{theorem}

Denote $\X_n^i = \X_n \backslash \{X_i\}$ and $S(X_i,\X_n) = \#(\X_n \cap C(X_i,R_{i,n}))^q \ones\{R_{i,n}\geq M\}$. We will apply Theorem \ref{massart} to $F(\X_n) = \sum_{i=1}^n S(X_i,\X_n)$. The quantity $(Z-Z_i')^2$ is bounded by $2(Z-Z_i)^2 + 2 (Z_i'-Z_i)^2$, where $Z_i = F(\X_n^i).$ For most $X_j$'s, $S(X_j,\X_n) = S(X_j,\X_n^i)$, and therefore $V$ can be efficiently bounded. More precisely,
\begin{align}
E&[V^{p/2}] = E\left[\left(\sum_{i=1}^n E[(Z-Z_i')^2|\X_n]\right)^{p/2} \right] \nonumber \\
&\leq n^{p/2} E[(Z-Z_n')^p] \tag{by Jensen's inequality}\nonumber \\[5pt]
&\leq n^{p/2} 2^{p-1} E[(Z-Z_n)^p + (Z'_n-Z_n)^p] \nonumber \\[5pt]
&= n^{p/2} 2^p E[(Z-Z_n)^p ] \mbox{ as } (Z,Z_n) \sim (Z'_n,Z_n)\nonumber  \\[5pt]
&= 2^pn^{p/2} E\left[\left(S(X_n,\X_n) + \sum_{j=1}^{n-1} (S(X_j,\X_n)- S(X_j,\X_{n-1}))\right)^p\right] \nonumber \\
&\leq2^pn^{p/2} 2^{p-1}\left( E[S(X_n,\X_n)^{p}] + E\left[\left(\sum_{j=1}^{n-1} (S(X_j,\X_n)- S(X_j,\X_{n-1}))\right)^p\right] \right). \label{ineqMassart}
\end{align}
Fix $p=3$. Lemma \ref{lem:bound_num_points} shows that for $p\geq 1$, $B_p = \sup_n E[S(X_n,\X_n)^p] <\infty$. Define $Y_j = (S(X_j,\X_n) - S(X_j,\X_{n-1}))$. Denote $G_j$ the event that $X_n \in C(X_j,R_{j,n-1})$. If $G_j$ is not realized, then $Y_j=0$. Expanding the product, 
\begin{align}
E\Bigg[\bigg(\sum_{j=1}^{n-1} (S(X_j,\X_n) - &S(X_j,\X_{n-1}))\bigg)^3\Bigg] \leq \sum_{j_1,j_2,j_3} E\left[\ones\{ G_{j_1} \cap G_{j_2} \cap G_{j_3}\} Y_{j_1}Y_{j_2}Y_{j_3}\right] \nonumber \\
&\leq \sum_{j_1,j_2,j_3} P\left( G_{j_1} \cap G_{j_2} \cap G_{j_3}\right)^{1/q'}E\left[(Y_{j_1}Y_{j_2}Y_{j_3})^{p'}\right]^{1/p'}, \label{eq:bound_1}
\end{align}
where$ \frac{1}{p'} + \frac{1}{q'} = 1$ and $p'\geq 1$ is some quantity to be fixed later. \\[5pt]
$\bullet$\;\; We first bound $E\left[(Y_{j_1}Y_{j_2}Y_{j_3})^{p'}\right]$. If $Y_{j_1} \neq 0$, then $R_{j_1,n-1}>M$. Therefore, \[E\left[(Y_{j_1}Y_{j_2}Y_{j_3})^{p'}\right] \leq \sqrt{P(R_{j_1,n-1}>M) E\left[(Y_{j_1}Y_{j_2}Y_{j_3})^{2p'}\right]}.\]
Also, $E\left[(Y_{j_1}Y_{j_2}Y_{j_3})^{2p'}\right] \leq E\left[Y_{j_1}^{6p'}\right] \leq B_{6p'}$, as $0 \le Y_{j_1} \le S(X_{j_1},\X_n)$. Therefore, using inequality \eqref{eq:unifBoundR}:
\[E\left[(Y_{j_1}Y_{j_2}Y_{j_3})^{p'}\right] \leq c\exp(-c M^d).\]
$\bullet$\;\; We now bound the probability $P( G_{j_1} \cap G_{j_2} \cap G_{j_3})$.\\[5pt]
If $j_1=j_2=j_3$, then it is clear that $P( G_{j_1} \cap G_{j_2} \cap G_{j_3}) \leq c/n$. However, in the general case, the joint law of the different $R_{j_i,n-1}$s becomes of interest. To ease the notation, assume that $j_i=i$ and denote $R_{i,n-1}$ simply by $R_i$. Also, define $D_{ij}$ the distance between $X_i$ and $X_j$. The fact that inequality \eqref{eq:unifBoundR} still holds conditionally on $X_1,X_2$ and $X_3$, and with the joint laws of $R_1, R_2$ and $R_3,$ will be repeatedly used.

\begin{lemma}\label{lem:prob_bound} The following bound holds:
\begin{equation} P(R_1 \geq t_1, R_2\geq t_2, R_3 \geq t_3| X_1,X_2,X_3) \leq c \exp\big(-cn(t_1^d+t_2^d+t_3^d)\big)
\end{equation}
\end{lemma}
\begin{proof}
Suppose that $\max t_i = t_1$. Inequality \eqref{eq:unifBoundR} states that $P(R_1 \geq t_1) \leq c \exp(-cnt_1^d)$. Likewise, it is straightforward to show that a similar bound holds conditionally on $X_1,X_2$ and $X_3$. As $t_1^d \geq \frac{t_1^d+t_2^d+t_3^d}{3}$, the result follows.
  \end{proof}

Let us prove that $P( G_{1} \cap G_{2}) \leq c/n^2$. If the event is realized, then $X_n$ is in the intersection of $C(X_1,R_1)$ and $C(X_2,R_2)$. Therefore, this intersection is non empty and $D_{12} \leq \sqrt{d}(R_1+R_2) $. Hence,
\begin{align*}
 P( G_{1} \cap G_{2}) &\leq P(D_{12} \leq \sqrt{d}(R_1+R_2)  \mbox{ and } X_n \in C(X_{1},R_{1})  \cap 
 C(X_{2},R_{2})) \\
 &\leq 2 P(D_{12} \leq 2\sqrt{d}R_1 \mbox{ and } X_n \in C(X_{1},R_1)) \\
 &= 2 E\left[ \ones\{D_{12} \leq 2\sqrt{d}R_1\} P(X_n \in C(X_{1},R_1) | \X_{n-1})\right]\\
 &\leq 2c E\left[ \ones\{D_{12} \leq 2\sqrt{d}R_1\} R_1^d \right] \\
 &\leq 2c E\left[ \int_{\left( D_{12}/(2\sqrt{d}) \right)^d} P(R_{1,n-1}^d\geq t | X_1,X_2)dt \right] \\
  &\leq 2c E\left[ \int_{\left( D_{12}/(2\sqrt{d}) \right)^d} \exp(-cnt) dt \right] \\
  &\leq \frac{c}{n} E\left[ \exp(-cnD_{12}^d)\right] \\
  &= \frac{c}{n} \int_0^1 P(cnD_{12}^d \leq - \ln(t)) dt \\
  &\leq \frac{c}{n} \int_0^1 \frac{- \ln(t)}{cn} dt = \frac{c}{n^2}.\\
\end{align*}
Finally, we bound $P( G_{1} \cap G_{2} \cap G_3)$. If the event is realized, then 
\[
\begin{array}{ccc}
\left\lbrace \begin{array}{c}
D_{12} \leq \sqrt{d}(R_1+R_2) \\
 D_{23} \leq \sqrt{d}(R_2+R_3) \\ 
D_{13} \leq \sqrt{d}(R_1+R_3) 
\end{array} \right.
& \implies & 
\left\lbrace \begin{array}{c}
D_{12} \leq 2\sqrt{d}R_1 \text{ or } D_{12} \leq 2\sqrt{d}R_2 \\
D_{23} \leq 2\sqrt{d}R_2 \text{ or } D_{23} \leq 2\sqrt{d}R_3 \\
D_{13} \leq 2\sqrt{d}R_1 \text{ or } D_{13} \leq 2\sqrt{d}R_3
\end{array} \right.
\end{array}
\]
This last event is an union of eight events. Each of these event is either bounded by an event of the form ($D_{12} \leq 2\sqrt{d}R_1$ and $D_{13} \leq 2\sqrt{d}R_1$) (six events), or by an event of the form ($D_{12} \leq 2\sqrt{d}R_1$ and $D_{23} \leq 2\sqrt{d}R_2$ and $D_{13} \leq 2\sqrt{d}R_3$) (two events). Using this, we obtain
\begin{align*}
P( &G_{1} \cap G_{2} \cap G_3) \\
&\leq 6P(X_n \in C(X_1,R_1) \mbox{ and } D_{12} \leq 2\sqrt{d}R_1 \mbox{ and } D_{13} \leq 2\sqrt{d}R_1) \\
&  + 2P(X_n \in C(X_1,R_1) \mbox{ and } D_{12} \leq 2\sqrt{d}R_1 \mbox{ and } D_{23} \leq 2\sqrt{d}R_2 \mbox{ and } D_{13} \leq 2\sqrt{d}R_3) \\
&\leq c E[R_1^d\ones\{ D_{12} \leq 2\sqrt{d}R_1 \mbox{ and } D_{13} \leq 2\sqrt{d}R_1\} ] \\
& +c E[R_1^d \ones\{ D_{12} \leq 2\sqrt{d}R_1 \mbox{ and } D_{23} \leq 2\sqrt{d}R_2 \mbox{ and } D_{13} \leq 2\sqrt{d}R_3 \}] \\
&= c E\left[ \int_{\left(\frac{\max(D_{12},D_{13})}{2\sqrt{d}}\right)^d}^{\infty} P(R_1^d\geq u| X_1,X_2,X_3)du \right] \\
&+ cE\left[ \int_{\left(\frac{D_{12}}{2\sqrt{d}}\right)^d}^{\infty} P\left(\left.R_1^d\geq u \mbox{ and } R_2 \geq \frac{D_{23}}{2\sqrt{d}} \mbox{ and } R_3 \geq \frac{D_{13}}{2\sqrt{d}} \right| X_1,X_2,X_3 \right)du \right] \\
&\leq c E\left[ \int_{\left(\frac{\max(D_{12},D_{13})}{2\sqrt{d}}\right)^d}^{\infty} e^{-cnu}du \right] + cE\left[ \int_{\left(\frac{D_{12}}{2\sqrt{d}}\right)^d}^{\infty} e^{\left(-cn(u+D_{23}^d+D_{13}^d)\right)} du \right] \\
&=\frac{c}{n} E\left[e^{-cn\max(D_{12},D_{13})^d} \right] + \frac{c}{n} E\left[ e^{-cn(D_{12}^d+D_{23}^d+D_{13}^d)} \right] \\
&\leq 2\frac{c}{n} E\left[e^{-cn\max(D_{12},D_{13})^d} \right] = \frac{c}{n} \int_0^1 P\left(\max(D_{12},D_{13})^d \leq \frac{-\ln(t)}{cn} \right)dt \\
&\leq \frac{c}{n} \int_0^1 \left( \frac{-\ln(t)}{cn} \right)^2 dt = \frac{c}{n^3}.
\end{align*}

Finally, inequality \eqref{eq:bound_1} becomes
\begin{align*}
E[V^{3/2}] &\leq c n^{3/2}(\exp(-cM^d) + \exp(-cM^d)(n^{1-1/q'} + n^{2-2/q'} + n^{3-3/q'}) \\
&\leq c n^{3/2+3(1-1/q')}\exp(-cM^d).
\end{align*} 
Choose $p'=3/\varepsilon$ and apply Markov inequality to conclude.


\bibliographystyle{abbrv}
\bibliography{biblio}   

\end{document}